\documentclass[11pt]{article}

\usepackage{xcolor} 

\usepackage{enumerate}

\usepackage{hyperref}
\usepackage{footmisc}

\usepackage{amsmath}
\usepackage{amsthm}
\usepackage{amssymb}
\usepackage{amsfonts}
\usepackage{bbm}
\usepackage{nicefrac}
\usepackage{mathtools}

\usepackage{tikz}
\usepackage{color}
\usepackage{graphicx}
\usepackage{fullpage}
\usepackage{float}
\usepackage{wrapfig}
\usepackage{caption}
\usepackage{subcaption}
\usepackage{algpseudocode}
\usepackage{algorithm}
\usepackage{listings}

\numberwithin{equation}{section}

\let\Horig\H

\DeclareMathOperator{\E}{\mathbb{E}}
\DeclareMathOperator{\Var}{Var}

\newcommand{\norm}[1]{\left\Vert #1 \right\Vert}

\newcommand{\beqq}{ \begin{equation*} }
\newcommand{\eeqq}{ \end{equation*} }

\newcommand{\bse}{{\boldsymbol e}}
\newcommand{\bsu}{{\boldsymbol u}}

\newcommand{\RN}[1]{%
  \textup{\uppercase\expandafter{\romannumeral#1}}%
}

\DeclareMathOperator{\Tr}{Tr}

\def \C {\mathbb{C}}

\def \P {\mathbb{P}}
\def \R {\mathbb{R}}

\def \sphv {\mathbf{\sigma}}
\def \sphve {\sigma}
\def \Wg {A}
\def \Wgnm {M}
\def \sWg {V}
\def \sGOE {V^G}
\def \Wgint{H}
\def \scl {\dd \sigma_{scl}(x)}
\def \eg {\lambda}
\def \bnp {B}
\def \smallevent {\Omega_N}
\def \rWgnm {G}
\def \rWg {\hat{G}}
\def \expo {P}

\newtheorem{theorem}{Theorem}[section]
\newtheorem{proposition}[theorem]{Proposition}
\newtheorem{corollary}[theorem]{Corollary}
\newtheorem{lemma}[theorem]{Lemma}

\newtheorem{definition}[theorem]{Definition}

\theoremstyle{remark}
\newtheorem{remark}[theorem]{Remark}

\DeclareMathOperator{\SSK}{SSK}
\DeclareMathOperator{\CW}{CW}
\DeclareMathOperator{\TW}{TW}

\newcommand{\beq}{ \begin{equation} }
\newcommand{\eeq}{ \end{equation} }

\newcommand{\dd}{\mathrm{d}}
\newcommand{\ii}{\mathrm{i}}

\newcommand{\fe}{Q}
\newcommand{\tF}{\tilde{F}}
\newcommand{\es}{s}
\newcommand{\go}{\mathcal{G}_1}
\newcommand{\gt}{\mathcal{G}_2}

\newcommand{\Exp}{\mathbb{E}}

\DeclareMathOperator{\intc}{\mathbf{I}}

\DeclareMathOperator{\trans}{tran}
\newcommand{\cFt}{\mathcal{F}^{\trans}_N}
\DeclareMathOperator{\ferro}{ferro}
\newcommand{\cFf}{\mathcal{F}^{\ferro}_N}
\DeclareMathOperator{\para}{para}
\newcommand{\cFp}{\mathcal{F}^{\para}_N}

\begin{document}

\title{Ferromagnetic to paramagnetic transition in spherical spin glass}

\author{Jinho Baik\footnote{Department of Mathematics, University of Michigan,
Ann Arbor, MI, 48109, USA \newline email: \texttt{baik@umich.edu}}
\and Ji Oon Lee\footnote{Department of Mathematical Sciences, KAIST, Daejeon, 34141, Korea
\newline email: \texttt{jioon.lee@kaist.edu}}
\and 
Hao Wu\footnote{Department of Mathematics, University of Michigan,
Ann Arbor, MI, 48109, USA \newline email: \texttt{lingluan@umich.edu}}}


\maketitle
\begin{abstract}

We consider the spherical spin glass model defined by a combination of the pure 2-spin spherical Sherrington-Kirkpatrick Hamiltonian and the ferromagnetic Curie-Weiss Hamiltonian. 
In the large system limit, there is a two-dimensional phase diagram with respect to the temperature and the coupling strength. 
The phase diagram is divided into three regimes; ferromagnetic, paramagnetic, and spin glass regimes. 
The fluctuations of the free energy are known in each regime. 
In this paper, we study the transition between the ferromagnetic regime and the paramagnetic regime in a critical scale.
\end{abstract}
      
\section{Introduction} 


We consider a disordered system defined by random Gibbs measures whose Hamiltonian is the sum of a spin glass Hamiltonian and a ferromagnetic Hamiltonian. 
Depending on the strength of the coupling constant and the temperature, the system may exhibit several phases in the large system limit. 
The paper is concerned with the fluctuations of the free energy near the boundary between two phases known as ferromagnetic and paramagnetic regimes. 

Consider the sum of the pure $2$-spin spherical Sherrington-Kirkpatrick (SSK) Hamiltonian and the Curie-Weiss (CW) Hamiltonian. 
We call this sum the SSK+CW Hamiltonian. 
We denote the coupling constant by $J$ and the inverse temperature by $\beta$. 
We consider the random Gibbs measure with the SSK+CW Hamiltonian.
The focus of this paper is on the free energy. 

The limiting free energy was obtained non-rigorously by Kosterlitz, Thouless, and Jones \cite{Kosterlitz1976} in 1976. 
When $J=0$, this formula is the explicit evaluation of the Crisanti--Sommers formula \cite{crisanti1992sphericalp} (which was proved rigorously by Talagrand \cite{Talagrand2006s})  in the case of the pure $2$-spin SSK.  
The Crisanti--Sommers formula is the spherical version of the Parisi formula \cite{parisi1980sequence, Talagrand2006}.
The formula of Kosterlitz, Thouless, and Jones shows a two-dimensional phase transition: see Figure~\ref{phase_diagram_ssk}. 
The three regimes are determined by the condition that $\max\{ 1, \frac1{2\beta}, J\}$ is equal to $1$ (spin glass regime), $\frac1{2\beta}$ (paramagnetic regime) or $J$ (ferromagnetic regime). 
The limiting free energy is analytic with respect to both $\beta$ and $J$ in each regime, but not on the boundary.

        \setlength{\unitlength}{1cm}
        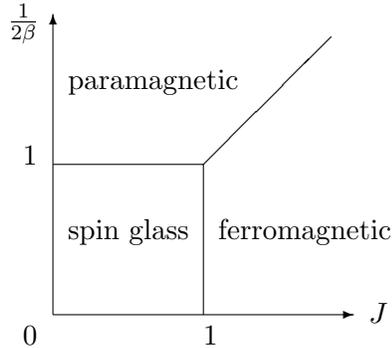
\begin{figure}
        \begin{center}
        \begin{picture}(5,5)
            
            \put(0,0){\vector(1,0){4}}
            \put(0,0){\vector(0,1){4}}
            \put(4.2,-0.1){\(J\)}
            \put(-0.6,3.8){\(\frac{1}{2\beta}\)}
            \put(-0.4,-0.4){0}
            \put(-0.4,2){1}
            \put(2,-0.4){1}
            \put(0.2,1){spin glass}
            \put(2,0){\line(0,1){2}}
            \put(0,2){\line(1,0){2}}
            \put(2,2){\line(1,1){1.7}}
            \put(0.2,3){paramagnetic}
            \put(2.2,1){ferromagnetic}
            
        \end{picture}
        \end{center}
            \caption{Phase diagram for SSK+CW model. Here, \(\beta\) is the inverse temperature and \(J\) is the coupling constant.}
            \label{phase_diagram_ssk}
        \end{figure}    

Recently, the authors of \cite{MR3649446} showed that the result of Kosterlitz, Thouless, and Jones is rigorous. 
Furthermore, the authors also evaluated the distribution of the fluctuations of the free energy in each regime.  (The case when $J=0$ was obtained earlier in \cite{MR3554380}.) 
The order of the fluctuations are $N^{-2/3}, N^{-1}, N^{-1/2}$ and the limiting distributions are Tracy-Widom, Gaussian, and Gaussian 
in the spin glass, paramagnetic regime, ferromagnetic regime, respectively. 
In the same paper, the transition between the spin glass regime and the ferromagnetic regime was also studied. 
However, the other two transitions and the triple point were left open. 
The goal of this paper is to describe the transition between the paramagnetic regime and and the ferromagnetic regime. 

\medskip

Another system which combines a spin glass and a ferromagnetic model is the SSK  with an external field. 
The difference between the CW Hamiltonian and an external field is that one is a quadratic function and the other is a linear function of the spin variables. 
These two models are related; see \cite{Chen2014} for a one-sided inequality.  
For the spin glass with external field, the fluctuations of the free energy were computed recently in \cite{ChenDeyPanchenkp2017, ChenSen2017} when the coupling constant is positive (for both SSK and SK (Sherrington-Kirkpatrick) cases with general spin interactions). 
However, the transitions are not obtained except for certain large deviation results \cite{FyodorovLeDoussal2014, DemboZeitouni2015}.
One of the interests of the SSK+CW model is that it is an easier model which can be analyzed in detail in the transitional regimes.

\subsection{Model}

Let \beq
    S_{N-1} = \{\sphv =(\sphv_1, \cdots, \sphv_N)\in \R^N: \sigma_1^2+\cdots + \sigma_N^2=N \}
\eeq
be a sphere in $\R^N$ of radius $\sqrt{N}$. 
Define the SSK+CW Hamiltonian by 
\begin{equation}
    H_N(\sphv) = H^{\SSK}_N(\sphv) + H^{\CW}_N(\sphv) ,  \qquad \sphv\in S_{N-1} 
\end{equation}
where 
\begin{equation}
    H_N^{\SSK}(\sphv) = \frac{1}{\sqrt{N}}\sum_{i,j = 1}^N\Wg_{ij}\sphve_i\sphve_j, 
    \qquad
    H_N^{\CW}(\sphv) =  \frac{J}{N} \sum_{i,j=1}^N \sphve_i \sphve_j = \frac{J}{N} \left( \sum_{i=1}^N \sigma_i\right)^2. 
\end{equation}
Here $J$ is the coupling constant.
The random coefficients $\Wg_{ij}$ satisfy $\Wg_{ij}=\Wg_{ji}$ and $\Wg_{ij}$, $i\le j$, are independent centered random variables. 
We call $\Wg_{ij}$ disorder variables. 
The precise conditions are given in Definition~\ref{def:Wgnm} below. 
Note that as a function of $\sphv$, $H_N^{\CW}(\sphv)$ is large when the coordinates of $\sphv$ have same sign. On the other hand, the maximizers $\sphv$ of $H_N^{\SSK}(\sphv)$ depend highly on $\{ \Wg_{ij}\}$.

With $\beta>0$ representing the inverse temperature, the free energy and the partition function are defined by 
\beq \label{eq:fepfde}
    F_N=\frac1{N} \log Z_N, \qquad Z_N= \int_{S_{N-1}} e^{\beta H_N(\sigma)}  \dd \omega_N(\sphv)
\eeq
where \(\omega_N\) is the normalized uniform measure on \(S_{N-1}\).
Note that $F_N$ and $Z_N$ are random variables since they depend on the disorder variables $\Wg_{ij}$. 
The free energy and the partition function depend on the parameters $\beta$ and $J$, 
\beq
    F_N=F_N(\beta, J), \qquad Z_N=Z_N(\beta, J).
\eeq

Since the Curie-Weiss Hamiltonian is a quadratic function of the spin variable, we can write the SSK+CW Hamiltonian as 
$H_N(\sphv)= \sum_{i,j=1}^N M_{ij} \sigma_i \sigma_j$ where $M_{ij} = \frac1{\sqrt{N}} \Wg_{ij} + \frac{J}{N}$
are non-centered random variables. In terms of matrix notations, 
\beq\label{H_ssk+cw}
    H_N(\sphv)= \sigma^T \Wgnm\sigma , \qquad \Wgnm= \frac1{\sqrt{N}} \Wg +  \frac{J}{N} \mathbf{1} \mathbf{1}^T
\eeq
with $\Wg=(\Wg_{ij})_{1\le i,j\le N}$, $\mathbf{1} = (1, \cdots, 1)^T$, $\Wgnm=(\Wgnm_{ij})_{1\le i,j \le N}$, and $\sigma=(\sigma_1, \cdots, \sigma_N)^T$. 
The non-centered random symmetric matrix $\Wgnm$ is an example of a real Wigner matrix perturbed by a deterministic finite rank matrix. Such matrices are often called spiked random matrices. 
We will use the eigenvalues of spiked random matrices in our analysis of the free energy. 

\bigskip

We assume the following conditions on the disorder variables. 
\begin{definition}[Assumptions on disorder variables]\label{def:Wgnm}
                Let \(\Wg_{ij}\), $i\le j$, be independent real random variables satisfying the following conditions: 
                \begin{enumerate}               
                    \item[\(\circ\)] All moments of \(\Wg_{ij}\) are finite and \(\E[\Wg_{ij}] = 0\) for all $i\le j$. 
                    \item[\(\circ\)] For all \(i < j\), \(\E[\Wg_{ij}^2] = 1\), \(\E[\Wg_{ij}^3] = W_3\), and \(\E[\Wg_{ij}^4] = W_4\) for some constants \(W_3 \in \R\) and \(W_4 \geq 0\). 
                    \item[\(\circ\)] For all \(i\), \(\E[\Wg_{ii}^2] = w_2\) for a constant \(w_2 \geq 0\). 
                \end{enumerate}
            Set \(\Wg_{ij} = \Wg_{ji}\) for \(i > j\).  Let \(\Wg = (\Wg_{ij})_{i,j = 1}^N\) and we call it a Wigner matrix (of zero mean). 
\end{definition}

\begin{definition}[Eigenvalues of non-zero mean Wigner matrices] 
Let $\Wgnm$ be the $N\times N$ symmetric matrix defined in~\eqref{H_ssk+cw}. 
We call it a Wigner matrix of non-zero mean \footnote{In \cite{MR3649446}, we consider the case when the diagonal entries of $\Wgnm$ have mean $\frac{J'}{N}$ and the off-diagonal entries have mean $\frac{J}{N}$ where $J$ and $J'$ are allowed to be different. However, in this case, $\Wgnm =\frac{1}{\sqrt{N}} + \frac{J}{N}\mathbf{1} \mathbf{1}^T + \frac{J'-J}{N} I$ where $I$ is the identity matrix. This only shifts all eigenvalues by a deterministic small number. As we will see in Remark \ref{rem:ls_convfcn}, it is not more general than the case with $J' = J$.\label{ft:mean_shift}}. 
Its eigenvalues are denoted by 
            \beq
                \lambda_1\ge \lambda_2\ge\cdots \ge \lambda_N. 
            \eeq 
\end{definition}

We introduce the following terminology.

\begin{definition}[High probability event] We say that an \(N\)-dependent event \(\Omega_N\) holds with high probability if, for any given \(D > 0\), there exists \(N_0 > 0\) such that 
\beqq
    \P(\Omega_N^c) \leq N^{-D}
\eeqq
for any \(N \geq N_0\).
\end{definition}

\subsection{Previous results in each regime}

We review the results on the fluctuations in each regime obtained in \cite{MR3649446}.
We state two types of results: one in terms of the eigenvalues of $\Wgnm$ and the other in terms of limiting distributions.

Set 
\beq
    \tilde{J}:= \max\{J, 1\}.
\eeq
It was shown in \cite{MR3649446} that the following holds with high probability. 
In both ferromagnetic and the spin glass regimes (given by $\tilde{J}>\frac1{2\beta}$), with any $\epsilon>0$, 
\beq \label{eq:thmspin1mid002}
    F_N = \tilde{F}_N+ \left(\beta - \frac1{2 \tilde{J} } \right) \left( \lambda_1 - \tilde{J}- \frac1{\tilde{J}} \right) + o(N^{-1+\epsilon}).
\eeq
In the paramagnetic regime (given by $\tilde{J}<\frac1{2\beta}$), 
\beq \label{eq:thmspin1high002} \begin{split}
    F_N = \tilde{F}_N - \frac{1}{2N} \sum_{i=1}^N \log\left( 2\beta+ \frac1{2\beta}- \lambda_i\right) + o(N^{-1}). 
\end{split} \eeq
Here, $\tilde{F}_N$ is a deterministic function of $N, \beta, J$. 
The above results show that the fluctuations of $F_N$ are determined, to the leading order, by the top eigenvalue $\lambda_1$ in the ferromagnetic and spin glass regimes, while they are determined by all eigenvalues in the paramagnetic regime. 

A limit theorem for $F_N$ follows if we use limit theorems for the eigenvalues of random matrices. 
The relevant random matrices are Wigner matrices of non-zero mean in~\eqref{H_ssk+cw}. 
For such random matrices, the following is known \cite{PRS2013, CDF2012} (see \cite{Baik-Ben_Arous-Peche05} for complex matrices): 
\beq \label{legld}
\begin{cases}
    N^{2/3} \left( \lambda_1 - 2  \right) \Rightarrow \TW_1  \quad &\text{if $J<1$, } \\
    N^{1/2} \left( \lambda_1 - J-\frac1{J}  \right)  \Rightarrow \mathcal{N}(W_3(J^{-2}-J^{-4}), 2(1-J^{-2})) \quad &\text{if $J>1$},
\end{cases} \eeq
where the convergences are in distribution. Here $\TW_1$ denotes the GOE Tracy-Widom distribution and $\mathcal{N}(a,b)$ denotes the Gaussian distribution of mean $a$ and variance $b$. 
The dichotomy is due to the effect of the non-zero mean; if $J$ is not large enough (i.e. $J<1$), then the influence of the non-zero mean is negligible to contribute to the fluctuations of the top eigenvalue. 
For $J<1$, the top eigenvalue is close to the second eigenvalue with order $O(N^{-2/3+\epsilon})$. But for $J>1$, the difference of the top eigenvalue and the second eigenvalue is of order $O(1)$. 

On the other hand, the following is also known (see Theorem 1.6 of \cite{MR3649446}): if a function $\varphi$ is smooth in an open interval containing the interval $[-2, \tilde{J} + \tilde{J}^{-1}]$, then 
\beq \label{lsts}
    \sum_{i=1}^N \varphi(\lambda_i)  -  N\int_{-2}^2 \varphi(x)\scl \Rightarrow \mathcal{N}(f,a), \qquad \scl := \frac{\sqrt{4 - x^2}}{2\pi} \dd x, 
\eeq
for some explicit constants $f, a$. 
This result is applicable to the paramagnetic regime. 

Together, we have the following asymptotic results obtained in Theorem 1.4 of \cite{MR3649446} (with a small correction in \cite{BaikLee17Erratum}):
\begin{enumerate}[(i)]
\item (Spin glass regime) If $\beta > \frac{1}{2}$ and $J < 1$, then 
\beq\label{eq:thmspin1low}
    \frac{1}{\beta-\frac{1}{2}} N^{2/3} \left( F_N - F \right) \Rightarrow \TW_{1}.
\eeq

\item (Paramagnetic regime) If $\beta < \frac{1}{2}$ and  $ \beta < \frac{1}{2J}$, then 
\beq\label{eq:thmspin1high}
    N \left( F_N - F \right) \Rightarrow \mathcal{N} \left(f_1, \alpha_1 \right). 
\eeq

\item (Ferromagnetic regime) If $J > 1$ and $\beta > \frac{1}{2J}$, then 
\beq \label{eq:thmspin1mid}
    \sqrt{N} \left( F_N - F \right) \Rightarrow \mathcal{N} \left(f_2', \alpha_2' \right). 
\eeq
\end{enumerate}
for some deterministic function $F=F(\beta, J)$ and some explicit constants $f_1, \alpha_1$, $f_2'$ and $\alpha_2'$ depending on \(\beta\) and \(J\).

\subsection{Results}

We state the results on the transition between the paramagnetic regime and the ferromagnetic regime. The boundary between these two regimes is given by the equation $\frac1{2\beta}=J$ with $J>1$. 
In the transitional regime, the correct scaling turns out to be the following: let $J>1$ be fixed and 
let $\beta=\beta_N$ be given by 
\begin{equation}
\label{eq:beta_N}
    2\beta = \frac{1}{J} + \frac{\bnp}{\sqrt{N}} 
\end{equation}
with fixed $\bnp\in \R$. 
The following is the first main result of this paper. 
This relates the free energy with the eigenvalues of $M$. 

\begin{theorem}
\label{thm:free_energy_para_ferro}
Let $\beta$ be given by~\eqref{eq:beta_N}. 
Then, for every  \(0<\epsilon < \frac18\), 
\begin{equation}
\label{eq:free_energy_para_ferro}
    F_N = \tF_N
    - \frac{1}{2N} \sum_{i = 2}^N g(\lambda_i) + \frac1{N} \fe(\chi_N) + O(N^{-3/2+4\epsilon}) , 
    \qquad  \chi_N:= \sqrt{N}(\lambda_1 - J+J^{-1}), 
\end{equation}
with high probability as $N\to \infty$, where 
\beq
    \tF_N = \beta (J+J^{-1})  -\frac{1}{2} - \frac{1}{2}\log(2\beta) + \frac{1}{N}\left( \frac14 \log N + \log\frac{\beta}{\sqrt{\pi}} \right),
    \quad
    g(z) := \log(J+J^{-1} - z).
\eeq
Also,  
\begin{equation}
\label{eq:def_fe}
    \fe(x)=  \frac{\es(x)}{2(\es(x) - x)} - \frac{\es(x)^2}{4(J^2 - 1)} + \frac{\log(\es(x) - x)}{2} + \log \intc\left(\frac{(\es(x)-x)^2}{J^2 - 1}\right)
\end{equation}
with 
\begin{equation}
\label{eq:def_es}
    \es(x) = \frac{x - \bnp(J^2 - 1) + \sqrt{(x + \bnp(J^2 - 1))^2 + 4(J^2 - 1)}}{2}
\end{equation}
and
 \begin{equation} \label{eq:ialpha_def}
    \intc(\alpha) = \int_{-\infty}^{\infty} \frac{e^{-\frac{\alpha}4 t^2+ \frac{\ii t}{2}}}{\sqrt{1 + \ii t}} 
    \dd t,
\end{equation}
where the square root denotes the principal branch. 
\end{theorem}

The formula~\eqref{eq:free_energy_para_ferro} shows a combined contribution from $\lambda_2, \cdots, \lambda_N$ and a distinguished contribution from $\lambda_1$. 
Compare the formula with~\eqref{eq:thmspin1mid002} and~\eqref{eq:thmspin1high002}.

\bigskip

Now we state a result analogous to~\eqref{eq:thmspin1high} and~\eqref{eq:thmspin1mid}. 
This follows if we have limit theorems for $\fe(\chi_N)$ and $\sum_{i = 2}^N g(\lambda_i)$. 
From the second part of~\eqref{legld}, $\fe(\chi_N)$ converges to an explicit function of a Gaussian random variable. 
On the other hand, $\sum_{i = 2}^N g(\lambda_i)$ is different from $\sum_{i = 1}^N g(\lambda_i)$ by one term. 
It is not difficult to show that removing one term does not affect the fluctuations much and the fluctuations are still given by a Gaussian random variable similar to~\eqref{lsts}; see Theorem~\ref{conj:1} in the next section. 
In random matrix theory, these sums are known as partial linear statistic and linear statistic, respectively. 
The main technical part of this paper is to evaluate the joint distribution of $\fe(\chi_N)$ and $\sum_{i = 2}^N g(\lambda_i)$. 
We show that jointly they converge in distribution to a bivariate Gaussian variable with an explicit covariance. 
See the next section for the precise statement. 
These results are interesting on their own in random matrix theory.
Putting together, we obtain the following result. 

\begin{theorem}
\label{thm:fluc_free_energy_para_ferro}
    We have
    \begin{equation}
        \begin{aligned}
        \label{eq:fluct_fe_para_ferro}
        N &\left(F_N - \frac{1}{4J^2} - \frac{\bnp}{2J\sqrt{N}} - \frac{\log N}{4N} - \frac{\bnp^2 J^2}{4N} \right) \Rightarrow 
        \go + \fe(\gt)
        \end{aligned}
    \end{equation}
    in distribution as $N\to\infty$ where $\go$ and $\gt$ are bivariate Gaussian random variables with 
    \begin{align} \label{gomeanvc}
        &\Exp[\go]= \frac14 \log (J^2-1)+ \frac{w_2 - 2}{4J^2} + \frac{W_4 - 3}{8J^4} + \log\frac{1}{2\sqrt{\pi}J} , \\
        &\Var[\go] =  
        -\frac{1}{2}\log(1 - J^{-2}) + \frac{w_2 - 2}{4J^2} + \frac{W_4 - 3}{8J^4}, 
    \end{align}
    \beq \label{gtmeanvc}
        \Exp[\gt]= W_3(J^{-2} - J^{-4}), \qquad \Var[\gt]= 2(1 - J^{-2}),
    \eeq
    and 
        \beq
         \mathrm{Cov}(\go, \gt) = \frac{W_3(J^{-2} - J^{-4})}{2}. 
    \eeq
    Note that $\go$ and $\gt$ do not depend on $B$. The function $\fe$ is defined in~\eqref{eq:def_fe}.
\end{theorem}

Note that if the third moment $W_3$  of $\Wg_{ij}$ with $i\neq j$ is zero, then $\go$ and $\gt$ are independent Gaussians.

The above result is consistent with the results on ferromagnetic and paramgnatic regimes if we let formally $B\to +\infty$ and $B\to -\infty$, respectively. 
One can show that when $B\to +\infty$, $Q(\gt)$ dominates $\go$. 
Furthermore, while $Q(\gt)$ is not Gaussian, upon proper normalization, it converges to a Gaussian as $B\to +\infty$. See Figure \ref{fig:qg2}. 
On the other hand, when $B\to -\infty$, the leading two terms of $Q(\gt)$ are constants and the random part is smaller than $\go$.  
See Section \ref{sec:large_t_behavior} for details.

\begin{figure}
\centering
\begin{subfigure}{0.3\textwidth}
\includegraphics[width = 0.9\textwidth]{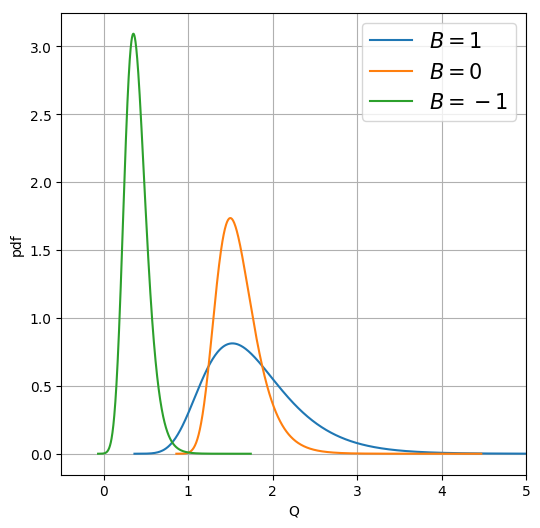}
\caption{pdf of $Q(\mathcal{G}_2)$}
\label{fig:non_gaussian}
\end{subfigure}
\begin{subfigure}{0.3\textwidth}
\includegraphics[width = 0.9\textwidth]{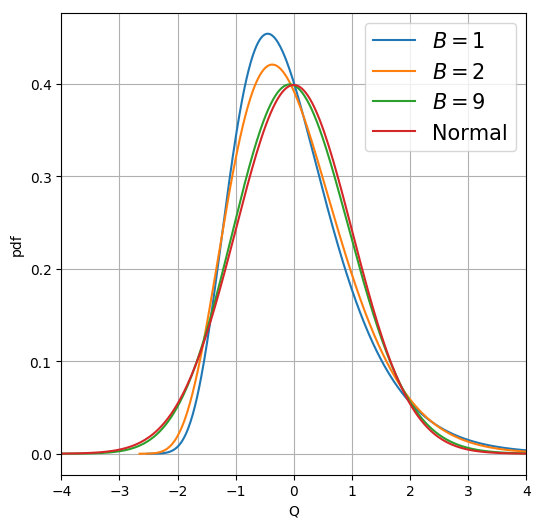}
\caption{pdf of normalized $Q(\mathcal{G}_2)$}
\label{fig:conv_to_gaussian}
\end{subfigure}
\caption{ 
(a) Probability density function of $Q(\mathcal{G}_2)$ for $B = -1,0,1$, 
(b) Probability density function of normalized $Q(\mathcal{G}_2)$  resembles a Gaussian density as $\bnp \rightarrow +\infty$.} 
\label{fig:qg2}
\end{figure}

\bigskip

Let us comment on the other transitions in the phase digram in Figure \ref{phase_diagram_ssk}.
As mentioned before, the transition between the spin glass and ferromagnetic regimes was discussed in \cite{MR3649446}. 
Note that \eqref{eq:thmspin1mid002} is valid in both regimes. 
It was shown that if we let $\beta > 1/2$ be fixed and consider $N$-dependent $J = 1 + wN^{-1/3}$, 
then for each $w \in \R$, \eqref{eq:thmspin1mid002} still holds.
Now, for such $J$, it was shown in \cite{BloemendalVirag2013} that $N^{2/3}(\eg_1 - 2) \Rightarrow \mathrm{TW}_{1,w}$
where $\mathrm{TW}_{1,w}$ is a one-parameter family of random variables interpolating $\TW$ and Gaussian distributions. 
Hence, we obtain the fluctuations for the transitional regime. 

On the other hand, the transition between the spin glass and paramagnetic regimes is an open question. 
By matching the fluctuation scales in both regimes, we expect that the critical scale is $\beta = \frac12 + O(\frac{\sqrt{\log N}}{N^{1/3}})$.

\subsection{Organization}

The rest of the paper is organized as follows. 
In Section~\ref{sec:RMT}, we first state new results on random matrices. They are given in Theorem \ref{conj:1} (partial linear statistics) and Theorem \ref{prop:asy_ind_ls_eg1} (joint convergence). Using them, we derive Theorem~\ref{thm:fluc_free_energy_para_ferro} from Theorem \ref{thm:free_energy_para_ferro}. 
In Section \ref{sec:free_energy}, we prove Theorem \ref{thm:free_energy_para_ferro}. 
In the next two sections, we prove the random matrix results stated in Section~\ref{sec:RMT}; Theorem \ref{conj:1} in Section \ref{sec:linear_statistics} and Theorem \ref{prop:asy_ind_ls_eg1} in Section \ref{sec:asy_gau}. 
In Section~\ref{sec:large_t_behavior}, we show that Theorem~\ref{thm:fluc_free_energy_para_ferro} is consistent with the previous results on ferromagnetic and paramagnetic regimes.

\subsubsection*{Acknowledgments}
The work of Jinho Baik was supported in part by NSF grant DMS-1664692 and the Simons Fellows program.
The work of Ji Oon Lee was supported in part by Samsung Science and Technology Foundation project number SSTF-BA1402-04.

\section{Results on Wigner matrices with non-zero mean} 
\label{sec:RMT}

In order to prove Theorem \ref{thm:fluc_free_energy_para_ferro} from Theorem \ref{thm:free_energy_para_ferro}, we need 
some new results on random matrices.
We need (i) a limit theorem for partial linear statistics \(\sum_{i = 2}^N g(\eg_i)\) and (ii) a joint convergence of the large eigenvalue and partial linear statistics. 
These results are interesting on their own in random matrix theory. 
We state them here and prove them in Section  \ref{sec:linear_statistics} and Section \ref{sec:asy_gau} below. 
Using these results, we prove Theorem \ref{thm:fluc_free_energy_para_ferro} in Subsection~\ref{sec:pfofth}. 

Recall that the $N\times N$ symmetric matrix $\Wgnm$ is given by $\Wgnm = \frac1{\sqrt{N}} \Wg +  \frac{J}{N} \mathbf{1} \mathbf{1}^T$
where $\Wg=(\Wg_{ij})$ is a symmetric matrix with independent entries for $i\le j$ satisfying the conditions given in Definition~\ref{def:Wgnm} and $\mathbf{1}=(1, \cdots, 1)^T$. 
The matrix $\Wgnm$ is called a Wigner matrix with a non-zero mean $\frac{J}{N}$. 
Recall that we assume  
\beq
    J>1.
\eeq
The eigenvalues of $\Wgnm$ are denoted by $\lambda_1\ge \cdots\ge \lambda_N$. 

It is known that $\lambda_1$ is close to $J+J^{-1}$ with high probability and $\lambda_2, \cdots, \lambda_N$ are in a neighborhood of $[-2, 2]$ with high probability. See Lemma~\ref{lemma:rigidity} below for the precise statement. 

\subsection{Partial linear statistics}

A linear statistic is the sum of a function of the eigenvalues. 
The fluctuations of linear statistics for Wigner matrices and other random matrix ensembles are of central interest in the random matrix theory; see, for example, \cite{MR1487983,Bai2005,Lytova2009}.
For Wigner matrices with non-zero mean, the following result was obtained in Theorem 1.6 and Remark 1.7 of \cite{MR3649446}. 
Set 
\beq
    \hat{J}=J+J^{-1}.
\eeq
Let \(\varphi: \R \rightarrow \R\) be a function which is analytic in an open neighborhood of \([-2, \hat{J}]\) and has compact support. 
Then, as $N\to \infty$, the random variable
\begin{equation} \label{eq:def_ls}
    \mathcal{N}_N(\varphi) := \sum_{i = 1}^N \varphi(\lambda_i) - N\int_{-2}^2\varphi(x)\mathrm{d}\sigma_{scl}(x)
    \Rightarrow \mathcal{N}(M(\varphi), V(\varphi))
\end{equation}
where 
\begin{equation} \label{meannoe}
\begin{aligned}
    M(\varphi) =& \frac{1}{4}(\varphi(2) + \varphi(-2)) - \frac{3}{2}\tau_0(\varphi) - J^{-1} \tau_1(\varphi) + (w_2 - 2)\tau_2(\varphi) \\
    & + (W_4 - 3)\tau_4(\varphi) + \varphi(\hat{J}) - \sum_{\ell= 2}^\infty J^{-\ell}\tau_\ell(\varphi) , \\
    V(\varphi) = & (w_2 - 2)\tau_1(\varphi)^2 + (W_4 - 3)\tau_2(\varphi)^2 + 2\sum_{\ell = 1}^\infty \ell\tau_\ell(\varphi)^2.
\end{aligned}
\end{equation}
Here, $W_4= \Exp[A_{12}^4]$, $w_2=\Exp[A_{11}^2]$, and 
\begin{equation}
    \tau_\ell(\varphi) = \frac{1}{\pi} \int_{-2}^2 \varphi(x)\frac{T_\ell(x/2)}{\sqrt{4 - x^2}} \dd x = \frac{1}{2\pi} \int_{-\pi}^{\pi} \varphi(2\cos(\theta))\cos(\ell\theta)\mathrm{d}\theta,
\end{equation}
where \(T_\ell(t)\) are the Chebyshev polynomials of the first kind.

\bigskip

We are interested in a partial linear statistic, $\sum_{i=2}^N \varphi(\lambda_i)$. 
See \cite{bao2013central, o2015partial} for other types of partial linear statistics. 
The partial linear static $\sum_{i=2}^N \varphi(\lambda_i)$ is the linear statistic minus one term $\varphi(\lambda_1)$.
Since $\lambda_1\to \hat{J}$ in probability (see the second part of \eqref{legld}), 
by~\eqref{eq:def_ls}, 
Slutsky's theorem implies that 
$$\sum_{i=2}^N \varphi(\lambda_i) - N\int_{-2}^2\varphi(x)\mathrm{d}\sigma_{scl}(x)\Rightarrow \mathcal{N}(M(\varphi)- \varphi(\hat{J}), V(\varphi)).
$$ 
Since this follows from~\eqref{eq:def_ls}, this is true assuming that $\varphi$ is analytic in an open neighborhood of \([-2, \hat{J}]\).
However, we are interested in the test function $\varphi(x)= g(x)= \log (\hat{J}-x)$ (see \eqref{eq:free_energy_para_ferro}). 
Since this function is not analytic at $x=\hat{J}$, the above simple argument does not apply. 
Nonetheless, if we adapt the proof of~\eqref{eq:def_ls}, one can show that it is enough to assume that the test function is analytic in a neighborhood of the interval $[-2,2]$, not of $[-2, \hat{J}]$.

\begin{theorem} \label{conj:1}
Let \(J > 1\). 
Then for every  test function \(\varphi\)
which is analytic in a neighborhood of \([-2, 2 ]\), 
\begin{equation}
\label{eq:def_partial_ls}
    \mathcal{N}_N^{(2)}(\varphi) := \sum_{i = 2}^N \varphi(\lambda_i) - N\int_{-2}^2 \varphi(x)\scl
    \Rightarrow \mathcal{N}(M^{(2)}(\varphi), V^{(2)}(\varphi))
\end{equation}
as $N\to \infty$ with
\begin{equation} \label{meannoett}
\begin{aligned}
    M^{(2)}(\varphi) =& \frac{1}{4}(\varphi(2) + \varphi(-2)) - \frac{3}{2}\tau_0(\varphi) - J^{-1} \tau_1(\varphi) + (w_2 - 2)\tau_2(\varphi) \\
    & + (W_4 - 3)\tau_4(\varphi)  - \sum_{\ell= 2}^\infty J^{-\ell}\tau_\ell(\varphi),
\end{aligned}
\end{equation}
and $V^{(2)}(\varphi) = V(\varphi)$  
where $V(\varphi)$ is defined in~\eqref{meannoe}. 
\end{theorem}

Note that
\beq
    M^{(2)}(\varphi) = M(\varphi)- \varphi(\hat{J})
\eeq
for $\varphi$ analytic in a neighborhood of $[-2, \hat{J}]$.

\begin{remark} \label{rem:ls_convfcn} 
We comment on a case when the test function depends on $N$. 
Consider the function $\varphi_N$ defined by  
$$
    \varphi_N(x) = \varphi(x) + \frac{\phi(x)}{N} + O(N^{-2})
$$
uniformly for $x$ in a neighborhood of $[-2, 2]$ for analytic functions $\varphi$ and $\phi$.
Define the corresponding linear statistic $\mathcal{N}^{(2)}_N(\varphi_N) = \sum_{i = 2}^N \varphi_N(\eg_i) - N\int_{-2}^2 \varphi_N(x)\scl$, then
\beq
\begin{aligned}
    \mathcal{N}^{(2)}_N(\varphi_N) = & \sum_{i = 2}^N \varphi_N(\eg_i) - N\int_{-2}^2 \varphi_N(x)\scl \\
    = & \mathcal{N}^{(2)}_N(\varphi) + \frac{1}{N}\left(\sum_{i = 2}^N \phi(\eg_i) - N\int \phi(x)\scl\right) +O(\frac{1}{N}).
\end{aligned}
\eeq
By Theorem \ref{conj:1}, the second order term converges to zero in probability. 
Thus, $\mathcal{N}^{(2)}_N(\varphi_N)$ and $\mathcal{N}^{(2)}_N(\varphi)$ converge to the same Gaussian distribution. 
The same argument also applies to full linear statistics; this is used in Remark \ref{re:coro:fluct_ind} below. 
Now, the claim in footnote\footref{ft:mean_shift} is verified by noting that $\varphi(x + \frac{J' - J}{N}) = \varphi(x) + \frac{\varphi'(x)(J' - J)}{N} + O(N^{-2})$. 
\end{remark}

\subsection{Joint convergence of the largest eigenvalue and linear statistics}

By Theorem~\ref{conj:1} and the second part of~\eqref{legld}, the partial linear statistic and the largest eigenvalue each converge to Gaussian distributions individually. 
The following theorem shows that they converge jointly to a bivariate Gaussian with an explicit covariance. 

\begin{theorem}
\label{prop:asy_ind_ls_eg1}
Let $J>1$. 
Then for $\varphi(x)$ which is analytic in a neighborhood of \([-2, 2 ]\),  \(\mathcal{N}_N^{(2)}(\varphi):= \sum_{i = 2}^N \varphi(\lambda_i) - N\int_{-2}^2 \varphi(x)\scl \) and \(\chi_N := \sqrt{N} (\lambda_1-\hat{J}) \) converges jointly in distribution to a bivariate Gaussian variable  
with mean 
\beq
\label{eq:bigau_mean}
    (M^{(2)}(\varphi), W_3(J^{-2} - J^{-4}))
\eeq
and covariance
\beq
\label{eq:bigau_cov}
    \begin{pmatrix}V^{(2)}(\varphi) & 2W_3\tau_2(\varphi)(1 - J^{-2}) \\ 2W_3\tau_2(\varphi)(1 - J^{-2}) & 2(1 - J^{-2}) \end{pmatrix}. 
\eeq
\end{theorem}

The proof of this theorem, given in Section \ref{sec:asy_gau}, is the main technical part of this paper.
We prove the theorem first for the Gaussian case, and then use an interpolation argument.

\subsection{Proof of Theorem \ref{thm:fluc_free_energy_para_ferro}} \label{sec:pfofth}

We now derive Theorem \ref{thm:fluc_free_energy_para_ferro} from Theorem \ref{thm:free_energy_para_ferro} using the results on the eigenvalues stated in the previous two subsections. 
The term $\fe(\chi_N)$ converges to $\fe(\gt)$ in distribution from Theorem~\ref{prop:asy_ind_ls_eg1}.  
Consider the rest. 
It was shown in (A.5) of \cite{MR3554380} that for $g(z) = \log(J+J^{-1}-z)$, 
\begin{equation}
\label{eq:mean_ls_log}
    \int g(z) \scl = \frac{1}{2J^2} + \log J. 
\end{equation}
Inserting $2\beta= J^{-1}+ \bnp N^{-1/2}$ and using the Taylor expansion \(\log(1 + \frac{\bnp J}{\sqrt{N}}) = \frac{\bnp J}{\sqrt{N}} - \frac{\bnp^2 J^2}{2N} + O(N^{-3/2})\), 
\beq
\label{eq:tf_ls}
    \tF_N - \frac12 \int g(z) \scl 
    = \frac{1}{4J^2} + \frac{\bnp}{2J\sqrt{N}} + \frac{\log N}{4N}
    + \frac1{N}  \left[ \frac{\bnp^2 J^2}{4} + \log \frac1{2\sqrt{\pi}J} \right] + O(N^{-3/2}). 
\eeq
We can evaluate $M^{(2)}(g)$ using (2.7) of \cite{MR3649446} which evaluated the $M(h)$ with $h(x)=\log(2\beta+\frac1{2\beta}-x)$: (note that $J'=J$ here)
\begin{equation}
    M^{(2)}(g) = \lim_{\beta \rightarrow \frac{1}{2J}} \left(M(h) - \log(2\beta+ \frac1{2\beta}- J-J^{-1})\right) 
    = -\frac12 \log(J^2-1)  - \frac{w_2-2}{2J^2} - \frac{W_4-3}{4J^4}. 
\end{equation}
The variance \(V^{(2)}(g)= V(g)\), which is independent of \(J\), is given by $4$ times (3.13) of \cite{MR3554380} if we replace $2\beta$ by $J^{-1}$: 
\beq
    V^{(2)}(g)= -2  \log(1-J^{-2}) + \frac1{J^{2}} (w_2-2)+ \frac1{2J^{4}} (W_4-3).
\eeq
For the covariance term, we have \(\tau_2(g) = - \frac{1}{2J^2}\) from (A.17) of \cite{MR3554380}. Hence, from Theorem \ref{conj:1} and \ref{prop:asy_ind_ls_eg1}, we obtain the result. 


\section{Proof of Theorem \ref{thm:free_energy_para_ferro}} 
\label{sec:free_energy}

The proof follows the steps for the proof of the Theorem 1.5 of \cite{MR3649446} for paramagnetic and ferromagnetic regimes with necessary adjustments. 
The analysis is based on applying a method of steepest-descent to a random integral. 
The location of the critical point is important. 
In the transitional regime, the critical point is close to the largest eigenvalue but not as close as the ferromagnetic case. 
On the other hand, the critical point is away from the largest eigenvalue in the paramagnetic case. See Subsection~\ref{subsec:steep_descent} below for details. 

\subsection{Preliminaries} 

The following formula is a simple result in \cite{Kosterlitz1976}. 

\begin{lemma}[\cite{Kosterlitz1976}; also Lemma 1.3 of \cite{MR3554380}]
\label{lemma:laplace_tran}
Let \(\Wgnm\) be a real \(N \times N\) symmetric matrix with eigenvalue \(\eg_1 \geq \eg_2 \geq \cdots \geq \eg_N\). Then for fixed \(\beta > 0\),
\begin{equation}
\label{eq:sph_line}
	\int_{S_{N-1}} e^{\beta  \sphv^T \Wgnm\sphv } \dd w_N(\sphv) = C_N \int_{\gamma - i\infty}^{\gamma +i\infty} e^{\frac{N}{2}G(z)}\dd z, \qquad G(z) = 2\beta  z - \frac{1}{N}\sum_{i =1}^N \log(z - \eg_i),
\end{equation}
where \(\gamma\) is any constant satisfying \(\gamma > \eg_1\), the integration contour is the vertical line from \(\gamma - \ii\infty\) to \(\gamma + \ii\infty\), the \(\log\) function is defined in the principal branch, and 
\begin{equation}
	C_N = \frac{\Gamma(N/2)}{2\pi \ii (N\beta)^{N/2 - 1}}.
\end{equation}
Here \(\Gamma(z)\) denotes the Gamma function.
\end{lemma}

Let $\Wgnm$ be a Wigner matrix with non-zero mean as in \eqref{H_ssk+cw}.
Then its eigenvalues $\lambda_i$ are random variables, and hence the above result gives a random integral representation of the partition function. 
In \cite{MR3649446, MR3554380}, the above random integral was evaluated using the method of steepest-descent for different choices of random matrices. 
The key ingredient in controlling the error term is a precise estimate for the eigenvalues which are obtained in the random matrix theory.

\begin{lemma}[Rigidity of eigenvalues: Theorem 2.13 of \cite{MR3098073} and Theorem 6.3 of \cite{KnowlesYin2011}]
\label{lemma:rigidity}
For each positive integer \(k \in [1, N]\), set \(\hat{k} := \min\{k, N + 1- k\}\). Let \(\gamma_k\) be the classical location defined by
\begin{equation}
	\int_{\gamma_k}^\infty \scl = \frac{1}{N}\left(k - \frac{1}{2}\right).
\end{equation}
Then, for every \(0< \epsilon < \frac1{2} \),
\begin{equation}
\label{eq:lemma:rigidity_1}
	\lvert\eg_k - \gamma_k \rvert \leq \hat{k}^{-1/3}N^{-2/3 + \epsilon}
\end{equation}
for all \(k = 2, 3, \cdots, N\) with high probability. Furthermore, for fixed \(J > 1\), recall $\hat{J} = J + J^{-1}$,
\begin{equation} \label{rigego}
	\lvert \eg_1 - \hat{J} \rvert \leq N^{-1/2 + \epsilon}
\end{equation}
holds with high probability.
\end{lemma}

From the rigidity, it is easy to obtain the following law of large numbers for eigenvalues. 

\begin{corollary}[c.f. Lemma 5.1 of \cite{MR3554380}]
\label{coro:law_of_large_number}
Fix $\delta > 0$, let \(\{f_\alpha\}_{\alpha \in I} \subset C^1[-2-\delta, 2 + \delta]\) be a family of monotonic increasing functions satisfying  \(\sup_{\alpha \in I} \max_x|f_\alpha(x)| \leq C_0\) and \(\sup_{\alpha \in I} \max_x|f_\alpha'(x)| \leq C_1\). Then, for every $0 < \epsilon < 1$, 
\begin{equation}
\label{eq:coro:law_of_large_number_2}
	\sup_{\alpha \in I}\left \vert \frac{1}{N}\sum_{i = 2}^N f_\alpha(\eg_i) - \int_{-2}^2 f_\alpha(x)\scl \right \vert 
	= O(N^{-1 + \epsilon})
\end{equation}
with high probability. 
\end{corollary}

\begin{proof}
Let $f=f_\alpha$ for some $\alpha\in I$. 
The absolute value on the left hand-side is bounded above by 
\begin{equation}
\begin{aligned}
	\left\vert \frac{1}{N} \sum_{i = 2}^N f(\lambda_i) - \frac{1}{N} \sum_{i = 2}^N f(\gamma_i) \right \rvert + \left \vert \frac{1}{N}\sum_{i = 2}^N f(\gamma_i) -  \int_{-2}^2 f(x)\scl \right \vert.
\end{aligned}
\end{equation} 
By Lemma~\ref{lemma:rigidity}, 
\begin{equation}
	\left \vert \frac{1}{N} \sum_{i = 2}^N (f(\lambda_i) - f(\gamma_i))\right\vert 
	\leq \frac{\max|f'(x)|}{N}\sum_{i = 2}^N |\lambda_i - \gamma_i| 
	\leq \frac{C_0}{ N^{1-\epsilon}}
\end{equation}
with high probability. On the other hand, set \(\hat{\gamma}_j\) by
\begin{equation}
	\int_{\hat{\gamma}_j}^2 \scl = \frac{j}{N}, \qquad j = 1, 2, \cdots,N,
\end{equation}
and by convention \(\hat{\gamma}_0 = 2\). As \(f(x)\) is a monotonic increasing function, for \(i = 2, 3, \cdots, N-1\),
\begin{equation}
	\int_{\hat{\gamma}_{i + 1}}^{\hat{\gamma}_i} f(x)\scl \leq \frac{1}{N}f(\gamma_i) \leq \int_{\hat{\gamma}_{i-1}}^{\hat{\gamma}_{i-2}} f(x)\scl.
\end{equation}
Thus, 
\begin{equation}
	\left\vert \frac{1}{N}\sum_{i = 2}^N f(\gamma_i) - \int_{-2}^2 f(x)\scl \right \vert \leq \frac{3\max|f(x)|}{N} \le \frac{3C_1}{N}.
\end{equation}
Since the upper bounds are independent of $f$, we obtain the result. 
\end{proof}

\subsection{Steepest-descent analysis}
\label{subsec:steep_descent}

We now apply steepest descent analysis to the integral in Lemma \ref{eq:sph_line}. 
We deform the contour to pass a critical point and show that the main contribution to the integral comes from a small neighborhood of the critical point. 
For \(G(z)\) given in \eqref{eq:sph_line}, it is easy to check that all solutions of $G'(z)=0$ are  real-valued, and there is a unique critical point $\gamma$ which lies in the interval  \( (\eg_1,\infty)\) (see Lemma 4.1 of \cite{MR3649446}).

Note that since $G$ is random, the critical point is also random. 
For the paramagnetic regime, it was shown in \cite{MR3649446} that $\gamma-\eg_1=O(1)$ with high probability. 
In the same paper, it was also shown that in the ferromagnetic regime, $\gamma-\eg_1=O(N^{-1 + \epsilon})$ with high probability. 
The following lemma establishes a corresponding result for the transitional regime; it shows that $\gamma-\eg_1=O(N^{-\frac{1}{2} + \epsilon})$ with high probability.

\begin{lemma}[Critical point]
\label{lemma:gamma_lambda_1}
Recall that (see \eqref{eq:beta_N}) \(J > 1\) is fixed and $2\beta =2\beta_N= \frac{1}{J} + \frac{\bnp}{\sqrt{N}}$ with fixed $\bnp \in \R$. 
Then, for every $0<\epsilon<\frac14$, 
\beq
\label{eq:gamma_lambda_1}
	\gamma= \eg_1 + \frac1{2\sqrt{N}} \left( -\chi_N - \bnp(J^2 - 1) + \sqrt{(\chi_N + (J^2 - 1)\bnp)^2 + 4(J^2 - 1)}
	\right) + O(N^{-1 + \epsilon})
\eeq
with high probability, where we set $\chi_N := \sqrt{N}(\eg_1-\hat{J})$.
\end{lemma}

Note that $\gamma$ given above is larger than $\lambda_1$ with high probability since the term in the big parenthesis is positive. 

\begin{proof}
Set 
\beq \label{eq:thettadefn}
	\theta:= \frac{-\chi_N - \bnp(J^2 - 1) + \sqrt{(\chi_N + (J^2 - 1)\bnp)^2 + 4(J^2 - 1)}}2. 
\eeq
Note that $\theta>0$. 
By the rigidity of \(\eg_1\), we have \(|\chi_N| \le N^{\frac{\epsilon}4}\) and hence, $\theta\le N^{\frac{\epsilon}3}$ with high probability. 
On the other hand, using $-a+ \sqrt{a^2 + b^2} = \frac{b^2}{\sqrt{a^2 + b^2} + a}$, 
\beqq
    \theta = \frac{2(J^2 - 1)}{\sqrt{(J^2 - 1)\bnp + \chi_N)^2 + 4(J^2 - 1)} + ((J^2 - 1)\bnp + \chi_N)},
\eeqq
and hence $\theta\ge CN^{-\frac{\epsilon}4}$ for some constant $C>0$ with high probability. Hence,
\beq \label{eq:thetesm}
	N^{-\frac{\epsilon}3} \le \theta \le N^{\frac{\epsilon}3}
\eeq
with high probability. Set
\beq \label{eq:gm1}
	\gamma_\pm :=  \eg_1 + \frac{\theta}{\sqrt{N}}\pm N^{-1+\epsilon}. 
\eeq
By the above properties of $\theta$, we have $\gamma_\pm>\lambda_1$ with high probability. 
We will show that $G'(\gamma_-) < 0$ and $G'(\gamma_+) > 0$ with high probability. 
Since $G'(z)$ is a monotone increasing function for real $z$ in the interval $(\lambda_1, \infty)$, this shows that $\gamma_-<\gamma<\gamma_+$ with high probability, proving the lemma. 

Recall that $\lambda_1\to \hat{J}$ in probability. 
Let us write
\beq \label{eq:gm2}
	\gamma_\pm =  J+\frac1{J} + \frac{\phi}{\sqrt{N}}\pm N^{-1+\epsilon}, \qquad \phi:= \theta+\chi_N
\eeq
where $\chi_N=\sqrt{N}(\eg_1-\hat{J})$. Note that $\phi= O(N^{\frac{\epsilon}3})$ with high probability. Now, notice that 
\begin{equation}
\label{eq:lemma:gamma_lambda_1_1}
	G'(z) = 2\beta - \frac{1}{N}\sum_{i = 2}^N\frac{1}{z - \lambda_i} - \frac{1}{N(z - \lambda_1)}.
\end{equation}
We apply Corollary \ref{coro:law_of_large_number} to the family of the function \(\{\frac1{z - x}\}_{z > 2 + c}\) for some constant $c>0$ and obtain 
\beqq
	G'(\gamma_\pm) = 2\beta - \frac{\gamma_\pm - \sqrt{\gamma_\pm^2 - 4}}{2} + O(N^{-1 + \frac{\epsilon}3}) - \frac{1}{N(\gamma_\pm - \lambda_1)} 
\eeqq
with high probability. 
By~\eqref{eq:gm2}, 
\beqq
	\frac{\gamma_\pm - \sqrt{\gamma_\pm^2 - 4}}{2} 
	= \frac1{J} - \frac1{J^2-1} \left( \frac{\phi}{\sqrt{N}}\pm N^{-1+\epsilon} \right) + O(N^{-1+ \frac{2\epsilon}3}).
\eeqq
By~\eqref{eq:gm1}, 
\beqq
	\frac{1}{N(\gamma_\pm - \lambda_1)} = \frac1{\theta\sqrt{N}} \left( 1\mp \frac{N^{-\frac12+\epsilon}}{\theta} 
	+ O \left( \frac{N^{-1+2\epsilon}}{\theta^2} \right)\right).
\eeqq
Using the formula of $2\beta$ and the estimate~\eqref{eq:thetesm} for $\frac{1}{\theta}$, we find that
\begin{equation}
\label{eq:s_quad}
	G'(\gamma_\pm) = \frac1{\sqrt{N}} \left( \bnp + \frac{\phi}{J^2-1} - \frac1{\theta}  \right)
	\pm \left( \frac{1}{J^2-1}+ \frac1{\theta^2} \right)  N^{-1+\epsilon} + O(N^{-1+ \frac{2\epsilon}3})
\end{equation}
with high probability since $0<\epsilon<\frac14$. 
By the definition of $\theta$, the leading term is zero. 
The coefficient of the second term is positive. Hence we find that $G'(\gamma_-)<0$ and $G'(\gamma_+)>0$, and we obtain the lemma. 
\end{proof}

Then we have the following lemma. 

\begin{lemma} \label{cor:s_chi_dis}
Set
\begin{equation} \label{eq:gamma_form}
	\text{$s=s_N:= \sqrt{N} (\gamma-  J-J^{-1})$ and $\Delta=\Delta_N:=\sqrt{N}(\gamma-\lambda_1)= s_N-\chi_N$.} 
\end{equation} 
Then, for every $\epsilon>0$, 
\begin{equation}
\label{eq:gamma_s}
	s = \frac{\chi_N - \bnp(J^2 - 1) + \sqrt{(\chi_N + (J^2 - 1)\bnp)^2 + 4(J^2 - 1)}}{2} + O(N^{-\frac12 + \epsilon})
\end{equation} 
with high probability. We also have 
\begin{equation} \label{eq:schboudn}
    \text{$|s|\le N^{\epsilon}$ and $N^{-\epsilon} \leq \Delta \leq N^{\epsilon}$}
\end{equation}
with high probability.
\end{lemma}

\begin{proof}
The previous lemma implies~\eqref{eq:gamma_s}. The first part of~\eqref{eq:schboudn} follows from the fact that $\chi_N=O(N^{\epsilon})$ with high probability. 
The second part is the estimate~\eqref{eq:thetesm} in the proof of the previous lemma. 
\end{proof}

We also need the following lemma. 
\begin{lemma} \label{lem:Fdre} 
For every $0<\epsilon<1$, 
\begin{equation} 
    	\frac1{N}\sum_{i=2}^N \frac{1}{(\gamma-\lambda_i)^2} = \frac{1}{J^2-1} + O(N^{-1+\epsilon})
\end{equation}
with high probability.
\end{lemma}

\begin{proof}
This follows from Corollary~\ref{coro:law_of_large_number} applied to $f(x)= \frac1{(\gamma-x)^2}$. 
\end{proof}

The following auxiliary lemma is used to estimate an error in the steepest descent analysis. 

\begin{lemma}
\label{lemma:ialpha_limit_behavior}
Define
\begin{equation}
	\intc_m(\alpha) := \int_{-\infty}^{\infty}  \frac{t^m}{\sqrt{1 + \ii t}} e^{-\frac{\alpha}{4}t^2 + \frac{\ii t}{2}} 
	\dd t
\end{equation}
for non-negative integers $m$ and $\alpha>0$, where the square root is the defined on the principal branch. We set \( \intc(\alpha) := \intc_0(\alpha)\); see~\eqref{eq:ialpha_def}. Then, 
\begin{equation}
\label{eq:ialpha_infty}
\intc(\alpha) = \sqrt{\frac{4\pi}{\alpha}}(1 + O(\alpha^{-1})) \quad \text{as  $\alpha \rightarrow +\infty$,} 
\end{equation}
\begin{equation}
\label{eq:ialpha_0}
	\intc(\alpha) = \sqrt{\frac{8\pi}{e}}(1 + O(\alpha)) \quad \text{as $\alpha \rightarrow 0_+$,} 
\end{equation}
and for every $m\ge 0$, 
\begin{equation}
	\text{$\intc_m(\alpha)$ is uniformly bounded for $\alpha\in (0,\infty)$.} 
\end{equation}
A particular consequence is that  the derivative \(\intc'(\alpha) = -\frac{1}{4} \intc_2(\alpha)\) is uniformly bounded for $\alpha>0$. 
Furthermore, $\intc(\alpha)>0$ for all $\alpha>0$.
\end{lemma}

\begin{proof}
Consider~\eqref{eq:ialpha_infty}.  
Applying the method of steepest-descent to $\intc(\alpha) = \int_{-\infty}^\infty g(t) e^{\alpha h(t)} \dd t$ with \(h(z) = -\frac{z^2}{4}\) and \(g(z) = \frac{1}{\sqrt{1 +\ii z}}e^{\frac{\ii z}{2}}\), we find that 
\begin{equation}
	\intc(\alpha) = \frac{e^{\alpha h(z_{c})} } {\sqrt{\alpha}}  \left[\sqrt{\frac{2\pi}{|h''(z_c)|}}g(z_c) + O(\alpha^{-1})\right] = \sqrt{\frac{4\pi}{\alpha}} (1 + O(\alpha^{-1}))
\end{equation}
as \(\alpha \rightarrow +\infty\).
For $\intc_m(\alpha)$, using $\int_{-\infty}^\infty y^m e^{-\alpha y^2} \dd y= O(\alpha^{-(m+1)/2})$, we find that 
\begin{equation} \label{eq:Ianpal}
	\intc_m(\alpha) = O(\alpha^{-\frac{m+1}{2}}) \quad \text{as $\alpha\to +\infty$.} 
\end{equation}

Consider the limit \(\alpha \rightarrow 0_+\). After the change of the variables \(t = z/\alpha\),
\begin{equation}
\label{eq:lemma_ialpha_1}
	\intc(\alpha) = \frac{e^{-\frac{1}{4\alpha}}}{\sqrt{\alpha}}\int_{-\infty}^{\infty} 
	\frac{ e^{-\frac{(z-\ii)^2}{4\alpha}} }{\sqrt{\alpha+\ii z}} 
	\dd z.
\end{equation}
The integrand is analytic in the complex plane minus the vertical line from $\ii \alpha$ to $\ii \infty$. Note that the saddle point is $\ii$ and it is on the branch cut. 
We show that the main contribution to the integral comes from the branch point $z=\ii \alpha$.
We deform the contour so that it consists of the following four line segments: $L_1$ from $\ii- \infty$ to $\ii$ on the left half-plane, $L_2$ from $\ii$ to $\ii \alpha$ lying on the left of the branch cut, $L_3$ from $\ii \alpha$ to $\ii$ lying on the right of the branch cut, and $L_4$ from $\ii$ to $\ii+\infty$ lying on the right-half plane. 
On $L_4$, setting $z=\ii + \sqrt{\alpha} x$, 
\beq
	\int_{L_4} \frac{ e^{-\frac{(z-\ii)^2}{4\alpha}} }{\sqrt{\alpha+\ii z}} \dd z
	= \sqrt{\alpha} \int_0^\infty \frac{e^{-\frac{x^2}{4}} }{ \sqrt{\alpha-1+\ii \sqrt{\alpha}x} } \dd x
	= O(\sqrt{\alpha})
\eeq	
as $\alpha\to 0$. Similarly, the integral over $L_1$ is also of the same order. 
On the other hand, setting $z=\ii \alpha+ \ii y$, 
\beq
	\int_{L_2\cup L_3} \frac{ e^{-\frac{(z-\ii)^2}{4\alpha}} }{\sqrt{\alpha+\ii z}} \dd z
	= 2 \int_0^{1-\alpha} \frac{e^{\frac{(\alpha +y-1)^2}{4\alpha}}}{\sqrt{y}} \dd y
	= 2 e^{\frac{(\alpha-1)^2}{4\alpha}} \int_0^{1-\alpha} \frac{e^{\frac{y}2+ \frac{y^2-2y}{4\alpha}}}{\sqrt{y}} \dd y.
\eeq
The function $y^2-2y$ decreases as $y$ increases from $y=0$ to $y=1$. Hence the main contribution to the integral comes near the point $y=0$. 
Using Watson's lemma, 
\beq
	\int_0^{1-\alpha} \frac{e^{\frac{y}2+ \frac{y^2-2y}{4\alpha}}}{\sqrt{y}} \dd y
	= \Gamma(1/2) \sqrt{2\alpha} (1+O(\alpha)). 
\eeq
Combining together and using $\Gamma(1/2)=\sqrt{\pi}$, we obtain \eqref{eq:ialpha_0}.
For $\intc_m(\alpha)$, the analysis is same except that we use 
\beq
	\int_0^{1-\alpha} (\ii \alpha+\ii y)^m \frac{e^{\frac{y}2+ \frac{y^2-2y}{4\alpha}}}{\sqrt{y}} \dd y
	= O(\alpha^{m+1/2}). 
\eeq
Hence, we find that for $m\ge 0$, $\intc_m(\alpha) = O(1)$ as $\alpha\to 0_+$.  Together with \eqref{eq:Ianpal}, this implies the uniform boundness of $\intc_m(\alpha)$.

For the positiveness of \(\intc(\alpha)\), we first write it as
\begin{equation}
\label{eq:lemma_ialpha_2}
	\intc(\alpha) = \int_{-\infty}^{\infty} \frac{e^{-\frac{\alpha}{4}t^2+ \frac{\ii}{2}(t - \arctan{t})}}{(1 + t^2)^{1/4}} \dd t = 2\int_{0}^\infty \frac{e^{-\frac{\alpha}{4}t^2}}{(1 + t^2)^{1/4}}\cos\left(\frac{1}{2}(t - \arctan{t})\right) \dd t .
\end{equation}
The function  \(\theta(t) = t - \arctan t\) is monotone increasing. 
We use the inverse function, $t=t(\theta)$, to change the variables and find that 
\begin{equation}
	\intc(\alpha)= 2\int_0^\infty e^{-\frac{\alpha}{4}t^2}\frac{(1 + t^2)^{3/4}}{t^2}\cos \left(\frac{\theta}{2} \right)\dd \theta, 
	\qquad t= t(\theta). 
\end{equation}
Since \(e^{-\frac{\alpha}{4}t(\theta)^2}\) is positive and monotone decreasing in \(\theta\), we obtain \(I(\alpha) > 0\) for every \(\alpha > 0\) if we show that (i)
\begin{equation}
	\int_0^\pi \frac{(1 + t^2)^{3/4}}{t^2}\cos\left(\frac{\theta}{2}\right)\dd \theta \geq -\int_\pi^{3\pi}\frac{(1 + t^2)^{3/4}}{t^2}\cos\left(\frac{\theta}{2}\right)\dd \theta,
\end{equation}
and (ii)
\begin{equation}
(-1)^k\int_{(2k - 1)\pi}^{(2k + 1)\pi}\frac{(1 + t^2)^{3/4}}{t^2}\cos\left(\frac{\theta}{2}\right)\dd \theta, \qquad k = 1,2,3,\cdots,
\end{equation}
is decreasing in \(k\). (i) can be verified numerically. On the other hand, (ii) follows immediately from the fact \((1 + t^2)^{3/4}/t^2\) is a decreasing function of \(t\). This completes the proof. 
\end{proof}

We now evaluate the integral in \eqref{eq:sph_line} using the steepest descent analysis.

\begin{lemma}
\label{lemma:steepest_descent}
Fix \(J > 1\) and let \(2\beta= J^{-1}+\bnp N^{-1/2} \). Consider G(z) defined in \eqref{eq:sph_line}.
Then, for every \(0< \epsilon <\frac18\), 
\begin{equation}
	\int_{\gamma - \ii\infty}^{\gamma + \ii\infty} e^{\frac{N}{2}G(z)} \dd z 
	= \frac{\ii \Delta e^{\frac{N}{2}G(\gamma)}}{\sqrt{N}}  \intc(F''(\gamma)\Delta^2) \left( 1 + O(N^{-\frac12+ 4\epsilon}) \right)
\end{equation}
with high probability, where
\begin{equation}
\label{eq:F_1}
	F(z) = 2\beta z - \frac{1}{N}\sum_{i = 2}^N \log(z - \lambda_i) - \frac{1}{N} \log(\gamma - \lambda_1) - \frac{z - \gamma}{N(\gamma - \lambda_1)}
\end{equation}
and \(\intc(\alpha)\) is defined in \eqref{eq:ialpha_def}. Recall that
$\Delta=\sqrt{N}(\gamma-\lambda_1)$ (see Lemma~\ref{cor:s_chi_dis}.) 
\end{lemma}

\begin{proof}
We choose the \(\gamma\), which defines the contour, as the critical point of \(G(z)\).
The path of steepest-descent is locally a vertical line near the critical point. 
It turns out that, instead of using the path of steepest-descent, it is enough to proceed the analysis using the straight line \(\gamma + \ii\R\) globally. 
This choice was also made for the analysis in the paramagnetic regime in \cite{MR3649446}. 

We first write, using the function $F(z)$, 
\begin{equation}
\begin{aligned}
\label{eq:lemma:steepest_descent_000}
	\int_{\gamma - \ii\infty}^{\gamma + \ii\infty} e^{\frac{N}{2}G(z)} \dd z = e^{\frac{N}{2}G(\gamma)}\int_{\gamma-\ii\infty}^{\gamma + \ii\infty} e^{\frac{N}{2}(G(z) -F(z)) + \frac{N}{2}(F(z) - G(\gamma))} \dd z .
\end{aligned}
\end{equation}
From the definitions of $G(z)$ and $F(z)$,  
\beq
	e^{\frac{N}{2}(G(z) -F(z))} = \sqrt{\frac{\gamma-\lambda_1}{z-\lambda_1}} e^{\frac{z-\gamma}{2(\gamma-\lambda_1)}}.
\eeq
Changing the variables \(z = \gamma + \ii t N^{-1/2}\) and using the notation $\Delta=\sqrt{N}(\gamma-\lambda_1)$, 
\begin{equation}
\begin{aligned}
\label{eq:lemma:steepest_descent_1}
	\int_{\gamma - \ii \infty}^{\gamma + \ii \infty} e^{\frac{N}{2}G(z)} \dd z 
	=& \frac{\ii e^{\frac{N}{2}G(\gamma)}}{\sqrt{N}}  \int_{-\infty}^{\infty} \frac{e^{\frac{\ii t}{2\Delta}} }{\sqrt{1 + \frac{\ii t}{\Delta}}}
	e^{\frac{N}{2}(F(\gamma + \ii t N^{-1/2}) - G(\gamma))} \dd t.
\end{aligned}
\end{equation}

It is easy to check that the part of the integral with $|t|\ge N^{\epsilon}$ is  small. 
To show this, we first note that
\beqq
	\Re \left( N\left(F(\gamma + \frac{\ii t}{\sqrt{N}}) - G(\gamma)\right) \right) 
	= - \Re \sum_{i = 2}^N \log \left( \frac{\gamma - \eg_i + \ii tN^{-1/2}}{\gamma - \eg_i} \right) 
	\le - \frac{N-1}2 \log \left( 1+ \frac{c^2t^2}{N}\right)
\eeqq
with high probability for some constant $c>0$, since there is a constant $c>0$ such that $c\le \gamma-\lambda_i \le \frac1{c}$ for all $i=2, \cdots, N$, with high probability. 
Hence, 
\begin{equation} \label{eq:IFNeb}
\begin{aligned}
		&\left| \int_{N^\epsilon}^\infty \frac{e^{\frac{\ii t}{2\Delta}}}{\sqrt{1 + \frac{\ii t}{\Delta}}} e^{\frac{N}{2}(F(\gamma + it N^{-1/2}) - G(\gamma)}\dd t \right|  
		\le \int_{N^\epsilon}^\infty e^{ - \frac{N-1}2 \log \left( 1+ \frac{c^2t^2}{N}\right)} \dd t  \\
    &\qquad \leq \int_{N^\epsilon}^N e^{-\frac{c^2}{8}N^{2\epsilon}} \dd t + \int_{N}^\infty \frac1{(c^2N^{-1}t^2)^{N/4}} \dd t = O(e^{-N^\epsilon}) + O(N^{-N/8})
\end{aligned}
\end{equation}
with high probability. 
 
Consider the part $|t|\le N^{\epsilon}$. 
Note that \(F(z)\) satisfies \(F(\gamma) = G(\gamma)\),  \(F'(\gamma) = G'(\gamma)= 0\), and for each $m\ge 2$,  \(F^{(m)}(z) = O(1)\) uniformly for $z$ in a small neighborhood of $\gamma$ (by Corollary \ref{coro:law_of_large_number}). 
For $m=2$, by Lemma~\ref{lem:Fdre}, 
\beq \label{eq:Fddesmb}
	c_1\le F''(\gamma)\le c_2
\eeq
for some constants $0 < c_1 < c_2$, uniformly in $N$. 
By Taylor expansion, for $|t|\le N^{\epsilon}$, 
\beq
	F(\gamma + \ii tN^{-1/2}) - G(\gamma)
	= - \frac{F''(\gamma) t^2}{2N} - \frac{\ii F'''(\gamma)  t^3}{6N^{3/2}} + O(N^{-2 + 4\epsilon})
\eeq
and hence,  
\beq
	e^{\frac{N}2 (F(\gamma + \ii tN^{-1/2}) - G(\gamma))} 
	= e^{-\frac{F''(\gamma)  t^2}{4}}\left( 1-  \frac{\ii F'''(\gamma) t^3}{12N^{1/2}} + O(N^{-1 + 6\epsilon}) \right) .
\eeq
Therefore, 
\begin{equation} \label{eq:InNes}
\begin{aligned}
	& \int_{-N^{\epsilon}}^{N^{\epsilon}} \frac{e^{\frac{\ii t}{2\Delta}} }{\sqrt{1 + \frac{\ii t}{\Delta}}}
	e^{\frac{N}{2}\left(F(\gamma + \ii t N^{-1/2}) - G(\gamma)\right)} \dd t \\
 &= \int_{-\infty}^{\infty}  \frac{e^{\frac{\ii t}{2\Delta}} }{\sqrt{1+ \frac{\ii t}{\Delta}}} e^{ - \frac{F''(\gamma)}{4}t^2 }  \dd t
 - \frac{\ii F'''(\gamma)}{12N^{1/2}}\int_{-\infty}^{\infty} t^3  \frac{e^{\frac{\ii t}{2\Delta}} }{\sqrt{1 + \frac{\ii t}{\Delta}}}  e^{ - \frac{F''(\gamma)}{4}t^2 } \dd t + O(N^{-1 + 6\epsilon}) \\
  &= \Delta \intc(F''(\gamma)\Delta^2) -  \frac{\ii F'''(\gamma) \Delta^{4}}{12 N^{1/2}} \intc_3(F''(\gamma)\Delta^2) + O(N^{-1 + 6\epsilon}).
 \end{aligned}
\end{equation}
By~\eqref{eq:Fddesmb} and Lemma \ref{cor:s_chi_dis}, $c_1 N^{-\epsilon}\le F''(\gamma)\Delta^2 \le c_2 N^{\epsilon}$. 
Hence, Lemma \ref{lemma:ialpha_limit_behavior} implies that 
\beq \label{eq:FFAAaa}
	\intc(F''(\gamma)\Delta^2) \ge c N^{-\epsilon}
\eeq
for some constant $c>0$. Hence, using Lemma \ref{cor:s_chi_dis}, Lemma \ref{lemma:ialpha_limit_behavior}, and the uniform boundedness of $F'''(\gamma)$, we find that~\eqref{eq:InNes} is equal to 
\begin{equation} \label{eq:InNesCL}
	 \Delta\intc(F''(\gamma)\Delta^2)(1 + O(N^{-\frac12+4\epsilon}))
\end{equation}
if $0<\epsilon<\frac18$. 
Thus, using~\eqref{eq:FFAAaa} and Lemma \ref{cor:s_chi_dis} again, we conclude that 
\begin{equation}
	\int_{\gamma - \ii \infty}^{\gamma + \ii \infty} e^{\frac{N}{2}G(z)} \dd z = \frac{\ii \Delta e^{\frac{N}{2}G(\gamma)}}{\sqrt{N}}\intc(F''(\gamma)\Delta^2)(1 + O(N^{-1/2+ 4\epsilon})).
\end{equation}
\end{proof}

\subsection{Proof of Theorem \ref{thm:free_energy_para_ferro}}

\begin{proof}[Proof of Theorem \ref{thm:free_energy_para_ferro}]
From Lemma \ref{lemma:laplace_tran} and Lemma \ref{lemma:steepest_descent}, for every \(0< \epsilon< \frac18\),
\begin{equation}
	Z_N = C_N \frac{\ii \Delta e^{\frac{N}{2}G(\gamma)}}{\sqrt{N}} \intc(F''(\gamma)\Delta^2)(1 + O(N^{-\frac12 + 4\epsilon}))
\end{equation}
with high probability. Using Stirling's formula,
\begin{equation}
	C_N = \frac{\Gamma(N/2)}{2\pi \ii (N\beta)^{N/2 - 1}} = \frac{\sqrt{N}\beta}{\ii \sqrt{\pi}(2\beta e)^{N/2}}(1 + O(N^{-1})),
\end{equation}
thus we find that \(F_N = \frac{1}{N}\log Z_N\) satisfies
\begin{equation}
\label{eq:F1}
	F_N = \frac{1}{2}(G(\gamma) - 1 - \log(2\beta))  + \frac{1}{N} \left(\log \left( \frac{\beta \Delta}{\sqrt{\pi}} \right)  + \log \intc(F''(\gamma)\Delta^2)  \right) + O(N^{-\frac32 + 4\epsilon})
\end{equation}
with high probability.

Let us consider \(G(\gamma)\). 
Since \(\gamma\) and \(\hat{J}= J+J^{-1}\) are away from \(\lambda_2, \cdots, \lambda_N\) with high probability, 
\begin{equation}\label{eq:thm:free_energy_beta_N_1}
\begin{split}
	\log(\gamma - \lambda_i) 
	&= \log(\hat{J} - \lambda_i)  - \log \left(1-  \frac{\gamma - \hat{J}}{\gamma - \lambda_i} \right) \\
	&= \log(\hat{J} - \lambda_i) + \frac{\gamma - \hat{J}}{\gamma - \lambda_i} + \frac{(\gamma - \hat{J})^2}{2(\gamma - \lambda_i)^2} + O(|\gamma-\hat{J}|^3) 
\end{split} 
\end{equation}
for \(i = 2, \cdots, N\), where we also use that $\gamma-\hat{J} = O(N^{-\frac12+\epsilon})$ with high probability (see Lemma~\ref{cor:s_chi_dis}). 
Then, using Lemma~\ref{lem:Fdre} and the fact that $G'(\gamma)= 2\beta - \frac1{N} \sum_{i=1}^N \frac1{\gamma-\lambda_i}=0$, 
\beqq \begin{split}
	&\frac{1}{N}\sum_{i = 2}^N \log(\gamma - \lambda_i) 
	= \frac{1}{N}\sum_{i = 2}^N \log(\hat{J} - \lambda_i) + 2\beta (\gamma-\hat{J})- \frac{\gamma - \hat{J}}{N(\gamma - \lambda_1)} + \frac{(\gamma - \hat{J})^2}{2(J^2-1)}  + O(N^{-\frac32+3\epsilon}) 
\end{split} \eeqq
with high probability. 
Hence, from the formula of \(G(z)\) in \eqref{eq:sph_line},
\beqq \begin{split}
 	G(\gamma)
	&= 2\beta \hat{J} - \frac{1}{N}\sum_{i = 2}^N \log(\hat{J} - \lambda_i)- \frac{1}{N}\log (\gamma - \lambda_1) + \frac{\gamma-\hat{J}}{N(\gamma-\lambda_1)} 
	- \frac{(\gamma - \hat{J})^2}{2(J^2-1)} + O(N^{-\frac32+3\epsilon}) \\
	&= 2\beta \hat{J} - \frac{1}{N}\sum_{i = 2}^N \log(\hat{J} - \lambda_i)- \frac{1}{N}\log \left( \frac{\Delta}{\sqrt{N}} \right) + \frac{s_N}{N\Delta} 
	- \frac{s_N^2}{2N(J^2-1)} + O(N^{-\frac32+3\epsilon})
\end{split} \eeqq
using the notations $s_N=\sqrt{N}(\gamma-\hat{J})$ and $\Delta=\sqrt{N}(\gamma-\lambda_1)$ in Lemma~\ref{cor:s_chi_dis}. 
Thus, 
\begin{equation}
\begin{aligned}
\label{eq:free_energy_exp}
	F_N =& \beta\hat{J} -\frac{1}{2} - \frac{1}{2}\log(2\beta) - \frac{1}{2N}\sum_{i = 2}^N \log(\hat{J} - \lambda_i) + \frac{1}{4N}\log N  \\
	& + \frac{1}{N} \left( \frac{s_N}{2\Delta}- \frac{s_N^2}{4(J^2-1)} + \frac12 \log \Delta + \log \frac{\beta}{\sqrt{\pi}} + \log \intc(F''(\gamma)\Delta^2)  \right)  +  O(N^{-\frac32 + 4\epsilon}). 
\end{aligned}
\end{equation}
To conclude Theorem \ref{thm:free_energy_para_ferro}, 
we use (i) the fact that $\Delta=s_N-\chi_N$, (ii) the asymptotic~\eqref{eq:gamma_s} of $s_N$ in terms of $\chi_N$,
(iii) the fact that $F''(\gamma)= \frac1{J^2-1} + O(N^{-1+\epsilon})$ which follows from Lemma~\ref{lem:Fdre}, and 
(iv) the fact that $\intc'(\alpha)$ is uniformly bounded for $\alpha>0$ (see Lemma~\ref{lemma:ialpha_limit_behavior}). 
\end{proof}

\section{Partial linear statistics}

This section is devoted to a proof of Theorem \ref{conj:1} on partial linear statistics. 
The proof is a simple modification of \cite{MR3649446} for the linear statistics of all eigenvalues, which, in turn, follows the proof of \cite{Bai2005, BaiSilverstein2010} for the case when the random matrix has zero mean.

\label{sec:linear_statistics}
\subsection{Proof of Theorem \ref{conj:1}} 

Recall $\hat{J} := J + J^{-1}$ denotes the classical location of the largest eigenvalue of a Wigner matrix of non-zero mean. Fix (\(N\)-independent) constants \(a_- < -2\) and \(2 < a_+ < \hat{J}\). Let \(\Gamma\) be the rectangular contour whose vertices are \((a_-\pm \ii v_0)\) and \((a_+ \pm \ii v_0)\) for some \(v_0\in (0,1]\). 
The contour is oriented counter-clockwise. 
For a test function $\varphi(x)$ which is analytic in a neighborhood of \([-2,2]\), we consider 
\begin{equation}
\begin{aligned}
	\mathcal{N}^{(2)}_N(\varphi) :=& \sum_{i = 2}^N \varphi(\eg_i) - N\int_\R \varphi(x)\scl \\
	=& \sum_{i = 2}^N \frac{1}{2\pi \ii}\oint_\Gamma \frac{\varphi(z)}{z - \lambda_i}\mathrm{d}z - \frac{N}{2\pi \ii} \int_\R \oint_\Gamma \frac{\varphi(z)}{z - x}\dd z\scl = -\frac{1}{2\pi i}\oint_\Gamma \varphi(z)\xi_N^{(2)} \dd z,
\end{aligned}
\end{equation}
where 
\begin{equation}
\label{eq:def_xi_N_2}
	\xi_N^{(2)} (z) : = \sum_{i = 2}^N \frac{1}{\eg_i - z} - N\int_\R \frac{1}{x - z}\scl.
\end{equation}
Decompose \(\Gamma\) into \(\Gamma_u \cup \Gamma_d \cup \Gamma_l \cup \Gamma_r \cup \Gamma_0\), where 
\begin{align}
	\Gamma_u =& \{z = x + \ii v_0: a_-\leq x \leq a_+\},\\
	\Gamma_d =& \{z = x - \ii v_0: a_-\leq x \leq a_+\},\\
	\Gamma_l =& \{z = a_- + \ii y: N^{-\delta}\leq |y| \leq v_0\},\\
	\Gamma_r =& \{z = a_+ + \ii y: N^{-\delta}\leq |y| \leq v_0\},\\
	\Gamma_0 =& \{z = a_- + \ii y: |y| < N^{-\delta}\} \cup \{z = a_+ + \ii y: |y| < N^{-\delta}\}
\end{align}
for some \(\delta > 0\). In the proof of Theorem 1.6 in \cite{MR3649446}, the authors showed that 
\begin{equation}
\label{eq:def_xi_N}
	\xi_N(z) := \sum_{i = 1}^N \frac{1}{\eg_i - z} - N \int_\R \frac{1}{x - z}\scl = \xi^{(2)}_N(z) + \frac{1}{\eg_1 - z}
\end{equation}
converges weakly to a Gaussian process with mean $b(z)= b^{(2)}(z)+ \frac{1}{\hat{J} - z}$ and covariance $\Gamma(z_i, z_j)=\Gamma^{(2)}(z_i, z_j)$ where $b^{(2)}(z)$ and $\Gamma^{(2)}(z_i, z_j)$ are given in the proposition below. 
Since for each fixed \(z \in \C_+\), \(\frac{1}{\lambda_1 - z} \rightarrow \frac{1}{\hat{J} - z}\) in probability (by Lemma \ref{lemma:rigidity}), it is natural to expect the following result for a partial sum.

\begin{proposition}
\label{prop:conv_xi_2}
Let
\begin{equation}
\label{eq:def_s}
	s(z) = \int \frac{1}{x - z} \scl = \frac{-z + \sqrt{z^2 - 4}}{2}
\end{equation}
be the Stieltjes transform of the semicircle measure. Fix a constant \(c > 0\) and a path \(\mathcal{K} \subset \C_+\) such that \(\Im z > c\) for  \(z \in \mathcal{K}\). Then the process \(\{\xi_N^{(2)}(z): z \in \mathcal{K}\}\) converges weakly to a Gaussian process with the mean 
\begin{equation}
\label{eq:mean_xi_2}
	b^{(2)}(z) = \frac{s(z)^2}{1 - s(z)^2}\left(-\frac{J}{1 + Js(z)} + (w_2 - 1)s(z) + s'(z)s(z) + (W_4 - 3)s(z)^3 \right) - \frac{1}{\hat{J} - z}
\end{equation}
and the covariance matrix 
\begin{equation}
\label{eq:cov_xi_2}
	\Gamma^{(2)}(z_i, z_j) = s'(z_i)s'(z_j)\left((w_2 - 2) + 2(W_4 - 3)s(z_i)s(z_j) + \frac{2}{(1 - s(z_i)s(z_j))^2}\right).
\end{equation}
\end{proposition}

\begin{remark}
	Note that as $z\to \hat{J}$, 
\beq
	\frac{s(z)^2}{1 - s(z)^2}\frac{J}{1 + Js(z)} = \frac{s'(z)}{\frac{1}{J} + s(z)} = \frac{1}{z - \hat{J}} + \frac{s''(\hat{J})}{s'(\hat{J})} + O(z - \hat{J}). 
\eeq
Hence,  \(b^{(2)}(z)\) is analytic near \(\hat{J}\) and thus analytic for \(z \in \C \setminus [-2,2]\).
\end{remark}

In order to complete the proof of Theorem \ref{conj:1}, we will prove the following lemma.

\begin{lemma}
\label{lemma:bound_xi_2}
	Define the events 
	\begin{equation}
		\smallevent := \{\eg_1 \geq \hat{J} - N^{-1/3}, \eg_2 \leq 2 + N^{-1/3}\}
	\end{equation}
	which satisfies \(\P(\smallevent^c) < N^{-D}\) for any fixed (large) \(D > 0\). Then for some \(\delta > 0\),
	\begin{equation}
	\label{eq:bound_xi_2}
		\lim_{v_0\rightarrow 0^+} \limsup_{N \rightarrow \infty} \int_{\Gamma_\#} \E |\xi_N^{(2)}(z)\mathbbm{1}_{\smallevent}|^2\mathrm{d}z = 0,
	\end{equation}
	where \(\Gamma_\#\) can be \(\Gamma_r\), \(\Gamma_l\) or \(\Gamma_0\).
\end{lemma}

From the explicit formulas \eqref{eq:mean_xi_2} and \eqref{eq:cov_xi_2}, it is easy to check that
\begin{equation}
\label{eq:bound_xi_limit}
	\lim_{v_0 \rightarrow 0^+}\int_{\Gamma_\#} \E|\xi^{(2)}(z)|^2 \mathrm{d}z = 0.
\end{equation}
Proposition \ref{prop:conv_xi_2}, Lemma \ref{lemma:bound_xi_2} and \eqref{eq:bound_xi_limit} imply that \(\mathcal{N}_N^{(2)}(\varphi)\) converges in distribution to a Gaussian random variable with the following mean and variance: 
\begin{equation}
	-\frac{1}{2\pi \ii}\oint_\Gamma \varphi(z)b^{(2)}(z)\mathrm{d}z, \qquad \frac{1}{(2\pi \ii)^2} \oint_\Gamma \oint_\Gamma \varphi(z_1)\varphi(z_2)\Gamma(z_1,z_2)\mathrm{d}z_1\mathrm{d}z_2. 
\end{equation}
It is direct to check that these are equal to \(M^{(2)}(\varphi)\) and \(V^{(2)}(\varphi)\) (see Section 4.2 in \cite{MR3649446}). 
We thus obtain Theorem \ref{conj:1}.

\subsection{Proof of Proposition \ref{prop:conv_xi_2}}

From Theorem 7.1 of \cite{Billingsley1968}, we need to show (i) the finite-dimensional convergence of \(\xi^{(2)}_N(z)\) to a Gaussian vector with desired mean and variance, and (ii) the tightness of \(\xi^{(2)}_N(z)\). We will base our proof on the corresponding properties of \(\xi_N(z)\) obtained in \cite{MR3649446}. Let us first recall the limit theorem for \(\xi_N(z)\).

\begin{lemma}[Proposition 4.1 in \cite{MR3649446}]
\label{lemma:ls_full}
Let \(s(z)\) and \(\mathcal{K}\) defined in the same way as in Proposition \ref{prop:conv_xi_2}. Then, the process \(\{\xi_N(z): z \in \mathcal{K}\}\) converges weakly to a Gaussian process \(\{ \xi(z): z\in \mathcal{K}\} \) with the mean 
\begin{equation}
	b(z) = \frac{s(z)^2}{1 - s(z)^2} \left(-\frac{J}{1 + Js(z)} + (w_2 - 1)s(z) + s'(z)s(z) + (W_4 - 3)s(z)^3\right)
\end{equation}
and the covariance matrix 
\begin{equation}
	\Gamma(z_i, z_j) = s'(z_i)s'(z_j)\left((w_2 - 2) + 2(W_4 - 3)s(z_i)s(z_j) + \frac{2}{(1 - s(z_i)s(z_j))^2} \right).
\end{equation}
\end{lemma}

Let \(z_1, z_2, \cdots, z_p\) are \(p\) distinct points in \(\mathcal{K}\). 
The above lemma implies that the random vector $(\xi_N(z_i))_{i = 1}^p$ 
converges weakly to a \(p\)-dimensional Gaussian distribution with the mean $(b(z_i))_{i = 1}^p$
and the covariance matrix \(\Gamma(z_i, z_j)\). 
Since the distance between $\mathcal{K}$ and $\eg_1$ is bounded below, \(\frac{1}{\eg_1 - z_i} \rightarrow \frac{1}{\hat{J} - z_i}\) in probability for \(i = 1, \cdots, p\). 
Hence, by Slutsky's theorem, $(\xi_N^{(2)}(z_i))_{i = 1}^p$
converges weakly to a \(p\)-dimensional Gaussian distribution vector with the mean \((b^{(2)}(z_i))_{i = 1}^p\) and the covariance matrix \(\Gamma^{(2)}(z_i, z_j)\), where
\begin{equation}
	b^{(2)}(z) = b(z) - \frac{1}{\hat{J} - z},
\end{equation} 
and \(\Gamma^{(2)}(z_i, z_j) = \Gamma(z_i, z_j)\).

From Theorem 12.3 of \cite{Billingsley1968}, in order to show the tightness of a random process \((\zeta_N(z))_{z \in \mathcal{K}}\), it is sufficient to show that (i) \((\zeta_N(z))_N\) is tight for a fixed \(z\), and (ii) the following H\"{o}lder condition holds: for some \(N\)-independent constant \(K > 0\),
\begin{equation}
	\E|\zeta_N(z_1) - \zeta_N(z_2)|^2 \leq K|z_1 - z_2|^2, \qquad z_1, z_2 \in \mathcal{K}.
\end{equation} 
In \cite{MR3649446}, the authors considered the random process \(\zeta_N(z):= \xi_N(z) - \E[\xi_N(z)]\), and proved that it satisfies conditions (i) and (ii). Now, we consider \(\xi_N^{(2)}(z) := \zeta^{(2)}_N + \E[\xi_N(z)]\), where \(\zeta^{(2)}_N(z) := \zeta_N(z) - \frac{1}{\lambda_1 - z}\). 
Since \(\E[\xi_N(z)]\) converges, it is enough to check that \((\zeta^{(2)}_N(z))_N\) satisfies conditions (i) and (ii). 
Now for a fixed \(z\),  the tightness of $(\zeta_N(z))_N$ and the boundedness of \(\frac{1}{\eg_1 - z}\) imply that \((\zeta^{(2)}_N(z))_N\) is tight. 
On the other hand, since \(\zeta_N(z)\) satisfies the H\"{o}lder condition and \(\Im z \geq c\) for \(z \in \mathcal{K}\),
\begin{equation}
\begin{aligned}
	\E|\zeta_N^{(2)}(z_1) - \zeta_N^{(2)}(z_2)|^2 
	&\leq 2\E|\zeta_N(z_1) - \zeta_N(z_2)|^2 + 2\E \left| \frac{1}{\eg_1 - z_1} - \frac{1}{\eg_1 - z_2} \right|^2 \\
	& \leq 2 K|z_1 - z_2|^2 + \frac{2|z_1 - z_2|^2}{c^{4}} = \left(K + \frac{2}{c^{4}} \right)|z_1 - z_2|^2.
\end{aligned}
\end{equation}
Thus \(\{\xi_N^{(2)}(z), z\in \mathcal{K}\}\) is tight. This completes the proof of Proposition \ref{prop:conv_xi_2}.

\subsection{Proof of Lemma \ref{lemma:bound_xi_2}}

For \(z \in \Gamma_0\), we notice that \(|\xi^{(2)}_N\mathbbm{1}_{\smallevent}|\leq CN\) and then
\begin{equation}
	\int_{\Gamma_0} \E|\xi^{(2)}_N\mathbbm{1}_{\smallevent}|^2 \leq CN^{2-\delta}.
\end{equation}
Thus \eqref{eq:bound_xi_2} holds for \(\Gamma_0\) with \(\delta > 2\). 
For \(\Gamma_r\) and \(\Gamma_l\), it is sufficient to show \(\E|\xi^{(2)}_N|^2 < K\) for some \(N\)-independent constant \(K > 0\). 
The authors in \cite{MR3649446} 
showed\footnote{Even though it is stated in Lemma 4.2 of \cite{MR3649446} that the lemma holds for sufficiently small $\delta > 0$, the proof of it is valid for any \(\delta > 0\), and we use \(\delta > 2\) for our purpose.} that  \(\E|\xi_N(z)|^2 < K\).
Hence, for \(z \in \Gamma_r\),
\begin{equation}
|\xi^{(2)}_N(z)\mathbbm{1}_{\smallevent}|^2 \leq 2|\xi_N(z)\mathbbm{1}_{\smallevent}|^2 + 2\left|\frac{1}{\eg_1 - z}\mathbbm{1}_{\smallevent}\right|^2. 
\end{equation}
The lemma then follows from the fact that \(|\frac{1}{\eg_1 - z}\mathbbm{1}_{\smallevent}|\) is bounded.

\section{Joint Distribution of \(\chi_N\) and \(\mathcal{N}_N^{(2)}(\varphi)\)}
\label{sec:asy_gau}

As before, let $\Wg$ be a random symmetric matrix of size $N$ whose entries are (up to the symmetry condition) independent centered random variables satisfying Definition \ref{def:Wgnm}. 
Let $\Wgnm= \frac1{\sqrt{N}} \Wg +  \frac{J}{N} \mathbf{1} \mathbf{1}^T$ where $J>1$. 
Let $\lambda_1\ge\cdots \ge \lambda_N$ be the eigenvalues of $\Wgnm$.

Let $\chi_N = \sqrt{N} (\lambda_1-\hat{J})$ denoting the rescaled largest eigenvalue. 
Given an analytic function $\varphi(x)$, recall the partial linear statistics \(\mathcal{N}_N^{(2)}(\varphi)= \sum_{i = 2}^N \varphi(\lambda_i) - N\int_{-2}^2 \varphi(x)\scl \).
We saw in the previous sections that $\chi_N$ and $\mathcal{N}_N^{(2)}(\varphi)$ converge individually to Gaussian random variables. 
In this section, we consider the joint distribution and prove Theorem \ref{prop:asy_ind_ls_eg1}. 
In Subsection \ref{subsec:asy_ind}, we first prove Theorem \ref{prop:asy_ind_ls_eg1} assuming that the disorder variables are Gaussian random variables.
In Subsection \ref{subsec:asy_gau}, the general disorder variables are considered using an interpolation trick.

\subsection{Asymptotic Independence for the GOE case}
\label{subsec:asy_ind}

Let the off-diagonal entries of $\Wg$ be Gaussian random variables of variance $1$ and the diagonal entries be Gaussian random variables of variance $2$. 
In random matrix theory, the random symmetric matrix $H = \frac1{\sqrt{N}} \Wg$ is said to belong to the Gaussian orthogonal ensemble (GOE).
A special property of GOE, compared with general random symmetric matrices, is that the probability measure of GOE is invariant under orthogonal conjugations. 

The following result is basically in \cite{Capitaine2009}. 

\begin{lemma}
\label{lemma:limit_thm_quad}
Let \((\frac{1}{\sqrt{2}}\Wg_{ii}, \Wg_{ij}, y_i)_{1\leq i < j \leq N}\) be i.i.d. standard Gaussian random variables.
Let  \( H = \frac{1}{\sqrt{N}} \Wg\) with $\Wg=(\Wg_{ij})_{1\le i,j\le N}$ and let \(Y =\frac{1}{\sqrt{N}} (y_1,\cdots, y_N)^T\).
Define \(G(z) = (H - zI)^{-1}\) for \(z \in \C \backslash [-2 - \delta,2 + \delta]\), which is well defined with high probability for fixed \(\delta > 0\). 
Then, for $z \in \R \backslash [-2-\delta, 2 + \delta]$,
\beq
\label{eq:def_quad}
    n_N(z) := \sqrt{N}(Y^*G(z)Y - \frac{1}{N}\mathrm{Tr}(G(z))) \Rightarrow n(z)， 
\eeq
where $n = n(z) := \mathcal{N}\left(0,2\int\frac{\scl}{(x - z)^2}\right)$ is a Gaussian random variable.
\end{lemma}

\begin{proof}[Proof of Lemma \ref{lemma:limit_thm_quad}]
We follow the idea presented in \cite{Capitaine2009}. 
By Theorem 5.2 of \cite{Capitaine2009}, it is enough to check the following three conditions for $G$: (i) There exists an $N$-independent constant $a$ such that $\norm{G} \leq a$ with high probability, (ii) $\frac{1}{N}\Tr G^2$ converges to a constant in probability, and (iii) $\frac{1}{N}\sum_{i = 1}^N G_{ii}^2$ converges to a constant in probability. 
They follow from rigidity of eigenvalue (Lemma \ref{lemma:rigidity}), law of large numbers (Corollary \ref{coro:law_of_large_number}), and local law (Theorem 2.9 of \cite{MR3098073}), respectively. 
\end{proof}

We are now ready to prove the following property of GOE matrices.  
\begin{proposition}
\label{prop:fluct_ind}
For $H$ defined in Lemma \ref{lemma:limit_thm_quad}, denote its eigenvalues by
$
    \rho_1 \geq \rho_2 \geq \cdots \geq \rho_N.
$
For fixed $k$, consider a random vector $(X_{N}^1, X^2_N, \cdots, X^k_N)$ whose entries are real measurable functions of those eigenvalues, i.e., $X_N^i = X_N^i(\rho_1, \rho_2, \cdots, \rho_N)$ for $i = 1,2,\cdots, k$. 
Suppose there is a random vector $(X^i)_{i = 1}^k$ such that $(X_{N}^i)_{i = 1}^k \Rightarrow (X^i)_{i = 1}^k$  as $N \rightarrow \infty$. 
Then for $n_N$ and $n$ defined as in \eqref{eq:def_quad}, $(X_{N}^1, X^2_N, \cdots, X^k_N, n_N) \Rightarrow (X^1, X^2\cdots, X^k, n)$, where $n$ is independent from $(X^1, X^2, \cdots, X^k)$.
\end{proposition}
\begin{proof}
For the convergence, it is enough to show (i) $(X_{N}^1, X^2_N\cdots, X^k_N, n_N)$ is tight, and (ii) convergence of characteristic function. 
The tightness follows from the tightness of individual random vector (variable), which is a consequence of individual convergence. 

For (ii), consider  the eigenvalue decomposition $H = OPO^T$, where $P = \mathrm{diag}(\rho_1, \rho_2, \cdots, \rho_N)$ and $O$ is an orthogonal matrix. 
Since the $H$ is orthogonal invariant, $P$ and $O$ are independent. 
Set $X = O^T Y$. Then \(X = \frac{1}{\sqrt{N}}(x_1,\cdots, x_N)\) where \(x_1, \cdots, x_N\) are i.i.d standard Gaussian (\(X\) is also independent with $P$).

Now, \(n_N = Y^*G(z)Y - \frac{1}{N}\mathrm{Tr}G(z)= \frac{1}{N}\sum_{i = 1}^N \frac{x_i^2 - 1}{\rho_i - z}\). 
Since $\E[ e^{t x_1^2}] = \frac1{\sqrt{1-2t}}$, we find that for any $t \in \ii\R$, the conditional expectation over $X$ given $P$ satisfies 
\beqq
    \E_X \big[ e^{ tn_N} \big| P \big] = \E_X \big[ e^{\frac{t}{\sqrt{N}}\sum_{i = 1}^N \frac{x_i^2 - 1}{\rho_i - z}} \big| P \big]
    = \prod_{i = 1}^N e^{-\frac{1}{2}\log (1 - \frac{2t}{\sqrt{N}(\rho_i - z)}) - \frac{t}{\sqrt{N}(\rho_i - z)}}.
\eeqq
Note that $(X^1_N, X^2_N, \cdots, X^k_N)$ only depends on the eigenvalues, and hence it is independent of $X$. 
Thus, for any $u_1, u_2, \cdots, u_k, t \in \ii\R$,
\begin{equation}
\label{eq:prop_fluct_ind_1}
\begin{aligned}
     \E \left[e^{\sum_{j = 1}^k u_jX^j_N + tn_N}\right] 
    =& \E \left[e^{\sum_{j = 1}^k u_jX^j_N}  \prod_{i = 1}^N e^{-\frac{1}{2}\log (1 - \frac{2t}{\sqrt{N}(\rho_i - z)}) - \frac{t}{\sqrt{N}(\rho_i - z_2)}}   \right].
\end{aligned}
\end{equation}
Since $ -\frac12 \log(1-2z) - z = z^2+ O(z^3) $ as $z\to 0$, using Corollary \ref{coro:law_of_large_number},
\beq
\label{eq:prop_fluct_ind_2} 
\begin{split}
    \prod_{i = 1}^N e^{-\frac{1}{2}\log (1 - \frac{2 t}{\sqrt{N}(\rho_i - z)}) - \frac{t}{\sqrt{N}(\rho_i - z)}}
    &= e^{ \frac1{N} \sum_{i=1}^N \frac{t^2}{(\rho_i - z)^2} + O(N^{-\frac{1}{2}}) } 
    = e^{t^2 \int \frac{1}{(x-z)^2} \scl + O(N^{-\frac{1}{2}}) } \\
    &=  \E\left[e^{tn(z)}\right] e^{O(N^{-\frac{1}{2}})} 
\end{split} \eeq
with high probability. Denote this high probability event by $\Omega_N$. Then,
\beq
\begin{aligned}
    \lim_{N \rightarrow \infty} \E \left[e^{\sum_{j = 1}^k u_jX^j_N + tn_N}\right] =& \lim_{N \rightarrow \infty} \left( \E \left[e^{\sum_{j = 1}^k u_jX^j_N + tn_N} \big |\Omega_N\right] \P(\Omega_N) + \E \left[e^{\sum_{j = 1}^k u_jX^j_N + tn_N}\big |\Omega_N^c\right] \P(\Omega_N^c) \right) \\
    =& \E\left[e^{\sum_{j = 1}^k u_jX^j}\right]\E\left[ e^{tn(z)}\right],
\end{aligned}
\eeq
since $t, u_1, u_2, \cdots, u_k\in \ii\R$ and hence all exponents are pure imaginary.
Note that the characteristic function of $(X^1, \cdots, X^k, n)$ is equal to the product of the characteristic functions of individual random vector (variable). 
Thus $n(z)$ is independent from $(X^1, \cdots, X^k)$. This completes the proof. 
\end{proof}
\begin{corollary}
\label{coro:fluct_ind}
Fix $\delta > 0$, consider $z_1 \in \C \backslash \R$ and $z_2 \in \R \backslash [-2-\delta, 2 + \delta]$. 
Recall $s(z)$ defined in \eqref{eq:def_s}. 
Then $(\Tr(G(z_1))-Ns(z_1), n_N(z_2))$ converges in distribution to independent Gaussian random variables. 
\end{corollary}

\begin{proof}
Note that $\Tr(G(z_1)) - Ns(z_1)$ is complex, we consider the random vector $(\Re(\Tr(G(z_1)) - Ns(z_1)), \Im(\Tr(G(z_1)) - Ns(z_1)))$. 
By Proposition \ref{prop:fluct_ind}, it is enough to show that $(\Re(\Tr(G(z_1)) - N(s_1)), \Im(\Tr(G(z_1)) - Ns(z_1))$ converges to a Gaussian random vector. 
Consider the expression $z_1 = E + i\eta$ for $\epsilon, \eta\in \R$ and $\eta \neq 0$. 
Recalling the definition of linear statistics $\mathcal{N}_N(\varphi)$ defined in \eqref{eq:def_ls}, we have 
$$
    \Re(\Tr(G(z_1) - Ns(z_1))) = \mathcal{N}_N(\varphi_r), \quad \ \varphi_r(x) = \frac{x - E}{(x - E)^2 + \eta^2},
$$
and 
$$
    \Im(\Tr(G(z_1) - Ns(z_1)) = \mathcal{N}_N(\varphi_i), \quad \ \varphi_i(x) = \frac{\eta}{(x - E)^2 + \eta^2}.
$$
That is, they are both linear statistics. Then Corollary then follows from Theorem 1.1 of \cite{Bai2005}.
\end{proof}

\begin{remark}
\label{re:coro:fluct_ind}
When we prove Theorem \ref{prop:asy_ind_ls_eg1} for GOE, we use Proposition \ref{prop:fluct_ind} and Corollary \ref{coro:fluct_ind} with $N$-dependent $z_i$. 
First, for a fixed $z_2 \in \R \backslash [-2-\delta, 2 + \delta]$ for some $\delta>0$, let $\tilde{z}_2 =\tilde{z}_2(N) := \sqrt{\frac{N + 1}{N}}z_2$.
Using the exactly same argument in the proof of Lemma \ref{lemma:limit_thm_quad}, one can show $n_N(\tilde{z}_2) \Rightarrow n(z_2)$. 
Since the \eqref{eq:prop_fluct_ind_2} still holds for $\tilde{z}_2$ and $n(z_2)$, the asymptotic independence in Proposition \ref{prop:fluct_ind} is still valid, i.e.
\beqq
    (X_{N}^1, X^2_N, \cdots, X^k_N, n_N(\tilde{z}_2)) \Rightarrow (X^1, X^2\cdots, X^k, n(z_2)),
\eeqq
where $n(z_2)$ is independent from $(X^1, X^2, \cdots, X^k)$.
Second, for $z_1 \in \C \backslash \R$, consider $\tilde{z}_1 = \tilde{z}_1(N) := \sqrt{\frac{N + 1}{N}}z_1$. Notice that
\beqq
    \frac{1}{x - \tilde{z}_1} = \frac{1}{x - z_1} + \frac{z_1}{2N(x - z_1)^2} + O(N^{-2}).
\eeqq
Then, by the discussion in Remark \ref{rem:ls_convfcn}, $\Tr(G(\tilde{z}_1)) - Ns(\tilde{z}_1) = \mathcal{N}_N(\frac{1}{x - \tilde{z}_1})$ converges to a Gaussian random variable. 
Now, putting together, for $\tilde{z}_1$ and $\tilde{z}_2$ defined as above, $(\Tr(G(\tilde{z}_1))-Ns(\tilde{z}_1), n_N(\tilde{z}_2))$ converge jointly to independent Gaussian random variables.
\end{remark}

We now prove Theorem \ref{prop:asy_ind_ls_eg1} for the case where the disorder belongs to GOE. 

\begin{proof}[Proof of Theorem \ref{prop:asy_ind_ls_eg1} when $A$ belongs to GOE]
Recall that $\lambda_i$ are the eigenvalues of $\Wgnm= \frac1{\sqrt{N}} \Wg +  \frac{J}{N} \mathbf{1} \mathbf{1}^T$ with $\Wg$ from the GOE. 
Since the means and variances follow from \cite{Capitaine2009} and Theorem \ref{conj:1}, it is enough to prove the asymptotic independence of $\chi_N$ and $\mathcal{N}_N^{(2)}(\varphi)$. 
(Notice that that \(W_3 = 0\) for Gaussian $A_{ij}$.)
Now, for any analytic test function $\varphi$, the partial linear statistics can be expressed as (see \eqref{eq:def_xi_N_2}) an integral of 
\begin{equation}
      \xi^{(2)}_N(z) =  \sum_{i = 2}^N \frac{1}{\eg_i - z} - N \int_\R \frac{1}{x - z}\scl, \quad z\in \C\backslash \R . 
\end{equation} 
Then according to Lemma \ref{lemma:bound_xi_2} and what follows, it is enough to prove that $\chi_N$ and $\xi_N^{(2)}(z)$ are asymptotically independent for fixed $z \in \C \setminus \R$.
Let $$\xi_N(z)= \xi_N^{(2)}(z) + \frac{1}{\eg_1 - z}  = \Tr(\Wgnm - zI)^{-1} - Ns(z).$$
Since $ \frac{1}{\eg_1 - z} \to \frac{1}{\hat{J} - z}$ in probability, it is enough to prove that $\chi_N$ and $\xi_N(z)$ are asymptotically independent.

Since the GOE is orthogonal invariant, for every deterministic matrix $U$, the eigenvalues of $\Wg+U$ have the same distribution as $\Wg+OU O^T$ for any orthogonal matrix $O$. 
Thus, we may consider the following equivalent model:
\begin{equation} \label{eq:MpertJ11}
    \Wgnm = \frac{1}{\sqrt{N}}\Wg + \mathrm{diag}(J, 0, \cdots,0). 
\end{equation}    
Following the proof of Theorem 2.2 in \cite{Capitaine2009}, we write 
\begin{equation}
\label{eq:block}
    \Wgnm = \begin{bmatrix} \frac{\Wg_{11}}{\sqrt{N}} + J & Y^* \\ Y & \hat{\Wgnm} \end{bmatrix}.
\end{equation}
Since $\det(\Wgnm-zI) = \det(\hat{\Wgnm}-zI) \left( \frac{\Wg_{11}}{\sqrt{N}}  +J - z - Y^*\hat{G}(z)Y \right)$ with 
\beq
    \hat{G}(z):= (\hat{\Wgnm} - zI_{N-1})^{-1}= ( \frac1{\sqrt{N}} \hat{\Wg} - zI_{N-1})^{-1}, 
\eeq 
the largest eigenvalue of $\Wgnm$ satisfies
\begin{equation}
    \eg_1 = J + \frac{\Wg_{11}}{\sqrt{N}} - Y^*\hat{G}(\eg_1)Y
\end{equation}
if $\lambda_1$ is not an eigenvalue of \(\hat{\Wgnm}\), which holds with high probability. 
Using the resolvent formula twice, we write 
\beqq \begin{split}
    \hat{G}(\lambda_1) 
    &= \hat{G}(\hat{J}) +( \hat{G}(\lambda_1) - \hat{G}(\hat{J}) )
    = \hat{G}(\hat{J}) + (\lambda_1-\hat{J}) \hat{G}(\lambda_1) \hat{G}(\hat{J}) \\
    &= \hat{G}(\hat{J}) + (\lambda_1-\hat{J}) \hat{G}(\hat{J})^2 + (\lambda_1-\hat{J})^2  \hat{G}(\lambda_1) \hat{G}(\hat{J})^2 . 
\end{split} \eeqq
Hence,
\beqq \begin{split}
    \lambda_1 - \hat{J} &=  \frac{\Wg_{11}}{\sqrt{N}} - \frac1{J}- Y^*\hat{G}(\eg_1)Y \\
    &= \frac{\Wg_{11}}{\sqrt{N}} - \frac1{J}- Y^*\hat{G}(\hat{J}) Y + (\lambda_1-\hat{J}) Y^* \hat{G}(\hat{J})^2 Y + (\lambda_1-\hat{J})^2  Y^* \hat{G}(\lambda_1) \hat{G}(\hat{J})^2 Y 
\end{split} \eeqq
with high probability. Moving all terms with factor $\lambda_1-\hat{J}$ to the left and taking it out as a common factor, we arrive at 
\begin{equation}
\label{eq:thm_fluct_ind_2}
    \chi_N = \sqrt{N}(\lambda_1 - \hat{J}) = \frac{\Wg_{11} - \sqrt{N}(\frac1{J} + Y^*\hat{G}(\hat{J})Y )}{1 + Y^*\hat{G}(\hat{J})^2Y + (\lambda_1 - \hat{J})Y^*\hat{G}(\lambda_1)\hat{G}    (\hat{J})^2Y} 
\end{equation}
with high probability. 

Note that \(\hat{\Wgnm}\) and \(Y\) satisfy the setting of Corollary \ref{coro:fluct_ind}  
up to the scaling factor \(\sqrt{\frac{N}{N-1}}\).
Set
\begin{equation}
\label{eq:thm_fluct_ind_1}
    \tilde{Y} = \sqrt{\frac{N}{N - 1}}Y, \qquad 
    \tilde{G}(z) = \big( \sqrt{\frac{N}{N - 1}} \hat{\Wgnm} - zI_{N - 1} \big)^{-1}
\end{equation}
Then, \(\tilde{Y}\) and \(\tilde{G}\) satisfy the setting of Corollary \ref{coro:fluct_ind}, and
\begin{equation}
    Y^*\hat{G}(\hat{J})Y = \sqrt{\frac{N - 1}{N}} \tilde{Y}^*\tilde{G}(\tilde{J})\tilde{Y}, 
    \qquad \tilde{J}:=  \sqrt{\frac{N}{N-1}} \hat{J}.
\end{equation}
Now, by Corollary \ref{coro:law_of_large_number}, 
\beq
\label{eq:thm_fluct_ind_5}
    \frac{1}{N - 1}\Tr(\tilde{G}(\tilde{J})) = s(\hat{J}) + O(N^{-1+\epsilon}) = -\frac{1}{J} + O(N^{-1+\epsilon})
\eeq
with high probability. By Lemma \ref{lemma:limit_thm_quad}, Corollary \ref{coro:law_of_large_number} and Lemma \ref{lemma:rigidity}, 
\beq
\label{eq:thm_fluct_ind_3}
    Y^*\hat{G}(\hat{J})^2 Y \rightarrow \frac{1}{J^2 - 1}, \qquad (\lambda_1 - \hat{J})Y^* \hat{G}(\lambda_1)\hat{G}(\hat{J})^2 Y \rightarrow 0 
\eeq
in probability. 
Using \eqref{eq:thm_fluct_ind_3}, \eqref{eq:thm_fluct_ind_5} and denoting the denominator in \eqref{eq:thm_fluct_ind_2} by $D_1$, we write
\begin{equation}    
\label{eq:thm_fluct_ind_6}
\chi_N = D_1^{-1} \left(\Wg_{11} - \tilde{n}_{N-1}(\tilde{J}) + O(N^{-\frac{1}{2} + \epsilon}) \right),
\end{equation}
where $n_{N - 1}(\tilde{J}) = \sqrt{N - 1}(\tilde{Y}^*\tilde{G}(\tilde{J})\tilde{Y} - \frac{1}{N - 1}\mathrm{Tr}(\tilde{G}(\tilde{J})))$ (see \eqref{eq:def_quad}) and $D_1 \rightarrow \frac{J^2}{J^2 - 1}$ in probability. Note that $\Wg_{11}$ and $n_{N - 1}(\tilde{J})$ are independent, the distribution of $\chi_N$ is governed by their convolution.

We now turn to the linear statistic $\xi_N(z)$. Using Schur complement of $\Wgnm$ with block structure in \eqref{eq:block}, for any \(z \in \C \backslash \R\),
\begin{equation}
\begin{aligned}
    \mathrm{Tr}(\Wgnm - zI)^{-1} =& (J + \frac{\Wg_{11}}{\sqrt{N}} - z - Y^*\hat{G}(z)Y)^{-1}(1 + Y^*\hat{G}(z)^2Y) + \mathrm{Tr}(\hat{G}(z))
\end{aligned}
\end{equation}
Using Lemma \ref{lemma:limit_thm_quad} and Lemma \ref{coro:law_of_large_number}, 
\beqq
    D_2 = D_2(N) :=  \frac{1 + Y^*\hat{G}(z)^2Y}{J + \frac{\Wg_{11}}{\sqrt{N}} - z - Y^*\hat{G}(z)Y} \rightarrow \frac{1 + s'(z)}{J - z - s(z)}
\eeqq
in probability. Then, by setting $\tilde{z} := \tilde{z}(N) = \sqrt{\frac{N}{N - 1}}z$, we write
\begin{equation}
\begin{aligned}
\label{eq:thm_fluct_ind_4}
    \xi_N(z) =& \mathrm{Tr}(\Wgnm - zI)^{-1} - Ns(z) = D_2 + \mathrm{Tr}\hat{G}(z) - Ns(z) + O(N^{-\frac12 + \epsilon}) \\
=& D_2 - \frac{s(z)}{2} + \frac{zs'(z)}{2} + \sqrt{\frac{N}{N-1}}\left(\mathrm{Tr}\tilde{G}(\tilde{z}) - (N-1)s(\tilde{z})\right) + O(N^{-\frac12 + \epsilon}).
\end{aligned}
\end{equation}
That is, the fluctuation of $\xi_N(z)$ is govern by $\mathrm{Tr}\tilde{G}(\tilde{z}) - (N-1)s(\tilde{z})$.  
Now using Corollary \ref{coro:fluct_ind} and Remark \ref{re:coro:fluct_ind}, one can conclude that $(\mathrm{Tr}\tilde{G}(\tilde{z}) - (N-1)s(\tilde{z}), n_{N - 1}(\tilde{J}))$ converge to independent Gaussian random variables. 
Furthermore, \(\Wg_{11}\) is independent of both \(Y\) and \(\hat{\Wgnm}\). 
Thus by \eqref{eq:thm_fluct_ind_6} and \eqref{eq:thm_fluct_ind_4}, \((\xi_N(z), \chi_N)\) converge to independent random variables. 
Theorem \ref{prop:asy_ind_ls_eg1} then follows. 
\end{proof}

\subsection{Proof of Theorem \ref{prop:asy_ind_ls_eg1} for general case}
\label{subsec:asy_gau}

We prove Theorem \ref{prop:asy_ind_ls_eg1} for general disorders, where the disorder matrix $\Wg$ is a Wigner matrix and satisfies Definition \ref{def:Wgnm}.
Unlike the GOE, Wigner matrices are not orthogonal invariant, hence we cannot apply~\eqref{eq:MpertJ11} where we replaced the rank-$1$ perturbation in $\Wgnm$ by a diagonal matrix.
To overcome the difficulty, we use an interpolation method. It has been successfully applied in many works in random matrix theory, where a given matrix and a reference matrix such as GOE are interpolated. We refer to \cite{Lytova2009} for its application in the analysis of linear eigenvalue statistics. 

Let $\sWg = \frac{1}{\sqrt{N}}A$ be a (normalized) Wigner matrix and $\sGOE$ be a (normalized) GOE matrix independent from $\sWg$. 
Define
\beq
  \Wgint(t) =  \sWg\cos t +  \sGOE\sin t
\eeq
so that $\Wgint(0) = \sWg$ and $\Wgint(\frac{\pi}{2}) = \sGOE$. Note that $\E[\Wgint_{ij}^2] = \frac{1}{N}$ for $i \neq j$. 
Let 
\beq
    \bse = \frac{1}{\sqrt N} \mathbf{1}^T = \frac{1}{\sqrt N} (1, 1, \dots, 1)^T \in \mathbb{R}^N 
\eeq
and
\beq
  \Wgnm(t) = \Wgint(t) + J \bse \bse^T,
\eeq
whose eigenvalues are denoted by $\eg_1 \geq \eg_2 \geq \cdots \geq \eg_N$. Define the resolvents
\beq \label{eq:resolvent}
  \rWgnm(z) = (\Wgnm-zI)^{-1}, \qquad \rWg(z) = (\Wgint-zI)^{-1}.
\eeq
Here, we omit the dependence on $t$ for the ease of notation.
We note that $\rWgnm$ and $\rWg$ are symmetric (not Hermitian).
For any (small) fixed \(\delta > 0\), $\hat{G}(z)$ is well-defined for $z \in \C \setminus [-2-\delta, 2+\delta]$ with high probability.

For $\chi_N = \sqrt{N}(\eg_1 - \hat{J})$, we notice that
\beq
  \rWg_{\bse \bse}(\lambda_1) := \langle \bse, \rWg (\lambda_1) \bse \rangle = -\frac{1}{J}
\eeq
with high probability. The claim holds since
\beq \begin{split}
  0= \det(\Wgnm - \lambda_1 I) &= \det (\Wgint - \lambda_1 I) \det(I + J \rWg(\lambda_1) \bse \bse^T) \\
  &= \det (\Wgint-\lambda_1 I) \det(I + J \bse^T \rWg(\lambda_1) \bse ) = \det (\Wgint-\lambda_1 I) \left( 1 + J \rWg_{\bse \bse}(\lambda_1) \right)
\end{split} 
\eeq
and $\lambda_1$ is not an eigenvalue of $\Wgint$ with high probability (See Lemma 6.1 of \cite{KnowlesYin2011}). Furthermore, by Taylor expansion,
\beq
  -\frac{1}{J} = \rWg_{\bse \bse}(\lambda_1) = \rWg_{\bse \bse}(\hat J) + \rWg'_{\bse \bse}(\hat J) (\lambda_1 - \hat J) + O(N^{-1+\epsilon})
\eeq
with high probability, since $|\lambda_1 - \hat J| = O(N^{-\frac{1}{2}+\epsilon})$  and $\| \rWg''(z) \| = O(1)$ with high probability. From the isotropic local law, Theorem 2.2 of \cite{KnowlesYin2011}, we find that
\beq
  \rWg_{\bse\bse}(\hat J) = s(\hat J) + O(N^{-\frac12 +\epsilon}), \qquad \rWg'_{\bse \bse}(\hat J) = s'(\hat J) + O(N^{-\frac{1}{2}+\epsilon})
\eeq
with high probability. Thus, using Lemma \ref{lemma:limit_thm_quad},
\beq
\label{eq:fluct_1_quad}
  \chi_N = \sqrt{N}(\lambda_1 - \hat J) = -\frac{\sqrt{N}(J^{-1} + \rWg_{\bse \bse}(\hat J))}{s'(\hat J)} + O(N^{-\frac{1}{2} + 2\epsilon})
\eeq
with high probability. That is, the behavior of $\chi_N$ is governed by the fluctuation of $\rWg_{\bse \bse}(\hat J)$.

To prove the Theorem \ref{prop:asy_ind_ls_eg1}, as in the Gaussian disorder case, it is enough to show the convergence of the joint distribution of $\chi_N$ and the full linear statistics $\xi_N(z) = \Tr(G(z)) - Ns(z)$ for fixed $z \in \C \setminus \R$.
Under the light of \eqref{eq:fluct_1_quad}, we set out to calculate the following characteristic function involving $\xi_N(z)$ and $\rWg_{\bse \bse}(\hat{J})$.
Explicitly, for \(t_1, t_2, t_3 \in \ii\R\) and \(z = E + \ii\eta\) with \(E \in \R\) and \(\eta > 0\), we define
\beq
  \E\left[e^{\expo(t)}\right] := \E \left[ e^{t_1 \Re \xi_N + t_2 \Im \xi_N + t_3 n_N} \right], \qquad \quad \expo (t) := t_1 \Re \xi_N(z) + t_2 \Im \xi_N(z) + t_3 n_N,
\eeq
where
\beq
  \qquad  n_N = \sqrt{N} \left( \rWg_{\bse \bse}(\hat J) + \frac{1}{J} \right).
\eeq
Note that \(n_N\) is real, the exponent \(\expo(t)\) is pure imaginary and thus \(|e^{\expo(t)}| \leq 1\). 
For our purpose, it is desired to estimate \(\E[e^{\expo (0)}]\). At $t=\frac{\pi}{2}$, the disorder \(\Wgint(\frac{\pi}{2})\) reduces to the GOE case. From Subsection \ref{subsec:asy_ind}, $\chi_N$ and $\xi_N$ are asymptotically independent in the GOE case, then
\beq
    \lim_{N \rightarrow \infty}\E\left[e^{\expo(\frac{\pi}{2})}\right] = \E\left[e^{t_1\Re{\xi} + t_2 \Im \xi} \right]\cdot \E\left[e^{t_3n}\right]
\eeq
for some Gaussian random variables \(\xi, n\) with known mean and variance.
Thus, it only remains to estimate the $t$-derivative of \(\E[e^{\expo(t)}]\). Here, we recall the following identity for the derivative of the resolvent \(\rWgnm\). For $i,j,a,b = 1,2,\cdots,N$,
\begin{equation}
\label{eq:res_deri}
  \frac{\partial}{\partial \Wgnm_{ij}}\rWgnm_{ab} = -\beta_{jk}(\rWgnm_{aj}\rWgnm_{kb} + \rWgnm_{ak}\rWgnm_{jb})
\end{equation}
with 
\begin{equation}
  \beta_{jk} = \begin{cases} 1 & j \neq k, \\ 1/2 & j = k.\end{cases} 
\end{equation}
We note that the above identity also holds if one replace \(\rWgnm\) by \(\rWg\). Thus for any fixed event $\Omega$,
\beq 
\label{eq:t_derivative}
\begin{split}
  \frac{\dd}{\dd t} \E \left[ e^{\expo(t)} \vert \Omega\right] &= \E \left[ \sum_{i\leq j} \frac{\dd \Wgnm_{ij}}{\dd t} \frac{\partial}{\partial \Wgnm_{ij}} e^{\expo(t)} \Bigg \vert \Omega\right] \\
    &= \sum_{i, j} \E \left[ \left(  \sWg_{ij}\sin t - \sGOE_{ij} \cos t\right) \left( t_1 \Re\left( \rWgnm^2\right)_{ij} + t_2 \Im \left(\rWgnm^2\right)_{ij} +\frac{t_3}{\sqrt N} \sum_{p, q} \rWg_{pi} \rWg_{jq} \right) e^{\expo(t)} \Big \vert \Omega\right].
\end{split} 
\eeq
The reason for the introduction of $\Omega$ will be revealed in a minute. The right hand side of \eqref{eq:t_derivative} motivates us to apply the generalized Stein's lemma. More precisely, we will use Proposition 3.1 of \cite{Lytova2009} with a small modification as follows:

\begin{proposition}
\label{prop:stein}
Given an event $\Omega$, let \(X\) be a random variable such that \(\E[|X|^{p + 2}\vert \Omega] < \infty\) for a certain non-negative integer \(p\). Denote the conditional cumulants of $X$ by $\kappa_l := \kappa_l(\Omega)$, $l = 1, \dots, p + 1$. Then for any function \(\Phi : \R \rightarrow \C\) of the class \(C^{p + 1}\) with bounded derivatives \(\Phi^{(l)}, l = 1, \dots, p + 1\), we have
\begin{equation} \label{eq:stein_main}
  \E[X \Phi(X) \vert \Omega] = \sum_{l = 0}^p \frac{\kappa_{l + 1}}{l!}\E[\Phi^{(l)}(X) \vert \Omega] + \epsilon_p,
\end{equation}
where the remainder term \(\epsilon_p\) admits the bound
\begin{equation}
\label{eq:stein_residue}
  |\epsilon_p| \leq C_p \E \left[|X|^{p + 2} \left( 1+ \max_{1\leq j\leq p+1} \left( \int_0^1 |\Phi^{(p+1)}(vX)| \dd v \right)^{\frac{p+2}{j}} \right)\Bigg \vert \Omega \right]
\end{equation}
for some constant $C_p$ that depends only on $p$.
\end{proposition}

\begin{proof}
We basically follow the proof of Proposition 3.1 of \cite{Lytova2009}. Let $\pi_p$ be the degree $p$ Taylor polynomial of $\Phi$ and let $r_p = \Phi - \pi_p$. Then, as in the proof of Proposition 3.1 of \cite{Lytova2009},
\beq
    \E[X \pi_p(X)\vert \Omega] = \sum_{j=0}^p \frac{\kappa_{j+1}}{j!} \E[\pi_p^{(j)}(X)\vert \Omega].
\eeq
Thus
\beq
    \left|\E[X \Phi(X)\vert \Omega] - \sum_{l=0}^p \frac{\kappa_{l+1}}{l!}\E[\Phi^{(l)}(X)\vert \Omega] \right| \leq \left| \E[X r_p(X)\vert \Omega]\right| + \sum_{l = 0}^p \frac{|\kappa_{l + 1}|}{l!}\left|\E \left[r_p^{(l)}(X)\vert \Omega\right] \right|.
\eeq
Since
\beq
    r_p(X) = \frac{X^{p+1}}{p!} \int_0^1 \Phi^{(p+1)}(vX) (1-v)^p \dd v,
\eeq
by the estimate $|\kappa_j| \leq (2j)^j \E[|X|^j\vert \Omega]$ and H\"older's inequality,
\beq
\begin{aligned}
\label{eq:prop:stein_1}
    \sum_{l = 0}^p \frac{|\kappa_{l + 1}|}{l!}\left|\E \left[r_p^{(l)}(X)\vert \Omega\right] \right| \leq &
    \sum_{l=0}^p \frac{\kappa_{l+1}}{l!(p-l)!} \E \left[ |X|^{p+1-l} \int_0^1 |\Phi^{(p+1)}(vX)| \dd v \Big \vert \Omega\right] 
    \\ \leq & \sum_{l=0}^p \frac{(2l+2)^{l+1}}{l!(p-l)!} \E \left[ |X|^{p+2} \left( 1 + \left( \int_0^1 |\Phi^{(p+1)}(vX)| \dd v \right)^{\frac{p+2}{p+1-l}} \right) \Bigg \vert \Omega\right].
\end{aligned}
\eeq
As $\left| \E[X r_p\vert \Omega]\right|$ can also be bounded by the right hand side of \eqref{eq:prop:stein_1}, the proof is complete.

\end{proof}
In order to apply Proposition \ref{prop:stein} to \eqref{eq:t_derivative}, we need prior bounds of $\expo(t)$ and its derivatives to bound $\epsilon_p$ in \eqref{eq:stein_main}.
As we will see later, it is enough to bound $\rWgnm_{ij}$, $(\rWgnm^2)_{ij}$, $\rWg_{ij}$ and $\sum_{p}\rWg_{ip}$.
In the following, we are going to introduce a high probability event $\Omega$, on which we have the desired bounds.

With the trivial bound \(\norm{\rWgnm} \leq \frac{1}{\eta}\) (recall that $z = E + i\eta$), we have that \(|\rWgnm_{ij}| \leq \frac{1}{\eta}\) and
$
  \left| (\rWgnm^2)_{ij}\right|\leq \norm{\rWgnm^2} \leq \frac{1}{\eta^2}.
$
For \(\rWg_{ij}\), we introduce the high probability event \(\Omega_1 = \{\lambda_1 \leq (2 + \hat{J})/2\}\). It is easy to check that \( \| \rWg \| \mathbbm{1}_{\Omega_1} \leq \frac{1}{\hat{J} - 2}\) and thus 
\beq \label{eq:hat_G_hat_J}
|\rWg_{ij}\mathbbm{1}_{\Omega_1}| \leq \frac{1}{\hat{J} - 2},
\eeq
For \(\sum_{p}\rWg_{ip}\), we recall the following concentration theorem for the quadratic function of \(\rWg\):

\begin{proposition}[Theorem 2.3 and Remark 2.4 of \cite{KnowlesYin2011}] 
\label{prop:iso_law}
Fix \(\Sigma \geq 3\). Set \(\varphi = (\log N)^{\log \log N}\). Then there exist constants \(C_1\) and \(C_2\) such that for any 
\[
  E \in [\Sigma, -2-\varphi^{C_1}N^{-\frac23}] \cup [2 + \varphi^{C_1}N^{-\frac23}, \Sigma],
\]
and any \(\eta \in (0,\Sigma]\), and any deterministic \(v, w \in \C^N\),
\beq
\label{eq:blinear_local_law}
  |\langle v, \rWg(z)w\rangle - s(z)\langle v, w \rangle| \leq \varphi^{C_2}\sqrt{\frac{\Im s(z)}{N\eta}}\norm{v}\norm{w}
\eeq
with high probability, uniformly on $z = E+\ii \eta$.
\end{proposition}

Let \(\bse_i := (0,\cdots, 1, \cdots, 0)\). Noting that \(\sum_{p = 1}^N \rWg_{pi} = \sqrt{N}\langle \bse, \rWg \bse_i \rangle\), we can derive a prior bound for $\sum_{p = 1}^N \rWg_{pi}$, which is summarized in the following Corollary.
\begin{corollary}
\label{coro:tail_bound}
For any fixed \(E \in \R \backslash [-2,2]\), the tail bound
\beq
\label{eq:coro:tail_bound_1}
  |\sum_{p}(\rWg(E))_{pi}| \leq N^\epsilon
\eeq
holds simultaneously for \(i = 1, \cdots, N\) with high probability. We also have that
\beq
\label{eq:coro:tail_bound_2}
  |\langle v, \rWg(E)w \rangle - s(E)\langle v, w\rangle| \leq \norm{v}\norm{w}N^{-\frac12 + \epsilon}
\eeq
with high probability.
\end{corollary}

\begin{proof}
We first prove \eqref{eq:coro:tail_bound_2}. Consider \(z = E + \ii N^{-1/2}\). Using Proposition \ref{prop:iso_law}, we find there exists some $C>0$ such that
\beq
\begin{aligned}
    \langle v, \rWg(E)w \rangle - s(E) \langle v, w \rangle| 
    & \leq |\langle v, (\rWg(z) - \rWg(E)w\rangle| + |\langle v, \rWg(z)w \rangle - s(z) \langle v, w \rangle| + |s(z) - s(E)||\langle v, w\rangle| 
    \\& \leq CN^{-1/2}\norm{v}\norm{w} + C\varphi^CN^{-\frac{1}{2}}\norm{v}\norm{w} + CN^{-1/2}.
\end{aligned}
\eeq
Here we also use the fact that \(\Omega_1\) holds with high probability.
Since $\varphi \ll N^{\epsilon}$, \eqref{eq:coro:tail_bound_2} then follows.
The tail bound \eqref{eq:coro:tail_bound_1} can be obtained from \eqref{eq:coro:tail_bound_2} by setting $v = \sqrt{N}\bse$ and $w = \bse_i$.
\end{proof}
Based on our discussion above, we are ready to introduce the high probability event as promised. 
Set \(s_1 := s(z)\), \(s'_1:=s'(z)\) and \(s_2 := s(\hat{J}) = -J^{-1}\), the desired high probability event \(\Omega\) is the intersection of \(\Omega_1\) and the following events:
\begin{align}
\label{eq:event_2} 
    &\Omega_2 = \{|\sum_p (\rWg(\hat{J}))_{pi}| \leq N^\epsilon,\quad \forall i= 1, \cdots, N\} \cap \{|\rWg_{\bse\bse}(\hat{J}) - s_2| \leq N^{-\frac{1}{2} + \epsilon}\}, 
\\ \label{eq:event_3} 
    &\Omega_3 = \{|\rWg_{ij} - \delta_{ij}s_2| \leq N^{-\frac{1}{2}+\epsilon}, \quad \forall i,j = 1,\cdots, N\},
\\ \label{eq:event_4} 
    &\Omega_4 = \{|G_{ij} - \delta_{ij}s_1|, |(\rWgnm^2)_{ij} - \delta_{ij}s'_1|\leq N^{-\frac{1}{2} + \epsilon}, \quad \forall i,j = 1,\cdots,N\},
\\ \label{eq:event_5}
    &\Omega_5 = \{ |V_{ij}|,|V^G_{ij}|,|M_{ij}| \leq N^{-\frac{1}{2} + \epsilon}, \quad \forall i,j = 1,\cdots,N\}.
\end{align}
Here, by Corollary \ref{coro:tail_bound}, $\Omega_2$ is a high probability event. The fact that $\Omega_3$ and $\Omega_4$ are high probability events can be checked from Theorem 2.8 and Theorem 2.9 of \cite{MR3098073}. It is easy to check that $\Omega_5$ is a high probability event from the existence of all moments. Furthermore, by the Lipshitz continuity of the resolvents, we also find that $\Omega$ holds uniformly on $t$ with high probability.

Applying Proposition \ref{prop:stein} to Equation \eqref{eq:t_derivative} conditioning on $\Omega$, we claim
\beq
\label{eq:t_derivative_stein}
\begin{aligned}
 & \sum_{i,j}\E \left[\sWg_{ij}\left(t_1 \Re(\rWgnm^2)_{ij} + t_2 \Im (\rWgnm^2)_{ij}  +\frac{t_3}{\sqrt N}\sum_{p, q} \rWg_{pi} \rWg_{jq}  \right) e^{\expo}\Big \vert \Omega \right] \\
 =& \sum_{l = 1}^3 \cos^lt\sum_{i,j} \frac{\kappa_{l  + 1}^{\sWg_{ij}}}{l!}\E\left[\left(\frac{\partial}{\partial \Wgnm_{ij}} \right)^l \left( \left( t_1 \Re(\rWgnm^2)_{ij}  + t_2 \Im (\rWgnm^2)_{ij}  +\frac{t_3}{\sqrt N}\sum_{p, q} \rWg_{pi} \rWg_{jq}\right) e^{\expo} \right) \Big \vert \Omega \right] + O(N^{-\frac12 +\epsilon})
\end{aligned}
\eeq
where \(\kappa_{l}^{\sWg_{ij}}\) denotes the \(l\)-th cumulant of \(\sWg_{ij}\). Here, it is legal to replace the conditional cumulants by $\kappa_l^{V_{ij}}$,  since $\Omega$ is a high probability event.

To prove the claim, we begin by controlling the remainder term $\epsilon_p$ in \eqref{eq:stein_main}. 
On \(\Omega\), \(\rWgnm_{ij}, \rWg_{ij}\) and \((\rWgnm^2)_{ij}\) are \( O(1) \), and 
\beqq
N^{-\frac12}\sum_{p,q} \rWg_{pi} \rWg_{jq} = N^{-\frac12}\left(\sum_p \rWg_{pi} \right) \left(\sum_q \rWg_{qj} \right)= O(N^{-\frac{1}{2} + \epsilon}).
\eeqq
Thus, \(\frac{\partial}{\partial \Wgnm_{ij}} P = O(1)\) on \(\Omega\). From the resolvent identity and the definition of event $\Omega$, we find $\norm{\rWgnm(z;vV_{ij}) - \rWgnm(z;V_{ij})} = O(N^{-\frac{1}{2} +\epsilon})$ for $0\leq v\leq 1$. Thus on $\Omega$, \(\frac{\partial}{\partial \Wgnm_{ij}} P(t;vV_{ij}) = O(1)\) for $0\leq v\leq 1$.
Furthermore, we notice that
\beq  
\label{eq:1d_rWgnm}
  \frac{\partial}{\partial \Wgnm_{ij}} (\rWgnm^2)_{ij} = \frac{\partial}{\partial \Wgnm_{ij}} \sum_k\rWgnm_{ki}\rWgnm_{jk} = - \beta_{ij}\left(2\rWgnm_{ij} (\rWgnm^2)_{ij} + \rWgnm_{ii}(\rWgnm^2)_{jj} + \rWgnm_{jj}(\rWgnm^2)_{ii}\right),
\eeq
and
\beq
\label{eq:1d_rWg}
  \frac{\partial}{\partial \Wgnm_{ij}} \sum_{p} \rWg_{pi} = - \beta_{ij}\left(\rWg_{ji}\sum_{p}\rWg_{pi} + \rWg_{ii}\sum_{p}\rWg_{pj}\right).
\eeq
Thus we can obtain similar estimates for higher derivatives of $P$. Since \(V_{ij}^5 = O(N^{-\frac52 +5\epsilon})\) on $\Omega_5$, we find that
\beq
    |V_{ij}|^5 \left( 1+ \max_{1\leq j\leq 5} \left( \int_0^1 \left|\left(\frac{\partial}{\partial \Wgnm_{ij}} \right)^5 P \right| \dd v \right)^{\frac{5}{j}} \right) \leq C N^{-\frac52 +C\epsilon}
\eeq
on \(\Omega\). 
That is, \( \epsilon_3 \leq C N^{-\frac52 +C\epsilon} \), and after summing over \(i, j\), the claim \eqref{eq:t_derivative_stein} is proved.

We next consider the term in \eqref{eq:t_derivative} containing \(\sGOE\). Noting that the cumulants of order higher than 2 vanish for Gaussian random variables, it reduces to
\beq
\begin{aligned} \label{eq:t_derivative_GOE}
 & \sum_{i,j}\E \left[\sGOE_{ij}\left(t_1 \Re(\rWgnm^2)_{ij} + t_2 \Im (\rWgnm^2)_{ij}  +\frac{t_3}{\sqrt N}\sum_{p, q} \rWg_{pi} \rWg_{jq}  \right) e^{\expo} \Big \vert \Omega\right] \\
 =& (\sin t) \sum_{i,j} \kappa_{2}^{\sGOE_{ij}}\E\left[\frac{\partial}{\partial \Wgnm_{ij}} \left( \left( t_1 \Re(\rWgnm^2)_{ij}  + t_2 \Im (\rWgnm^2)_{ij}  +\frac{t_3}{\sqrt N}\sum_{p, q} \rWg_{pi} \rWg_{jq}\right) e^{\expo} \right) \Big \vert \Omega\right] + O(N^{-\frac{1}{2} + \epsilon}),
\end{aligned}
\eeq
where \(\kappa_2^{\sGOE_{ij}}\) denotes the second cumulant of \( \sGOE_{ij} \). We now put \eqref{eq:t_derivative_stein} and \eqref{eq:t_derivative_GOE} into \eqref{eq:t_derivative} conditioning on $\Omega$. This yields
\beq \label{eq:I_l defined}
\frac{\dd}{\dd t} \E \left[ e^{\expo(t)} \right| \Omega] = (\sin t ) \sum_{l=1}^3 (\cos^l t) I_l - (\cos t \sin t) I_1^G + O(N^{-\frac{1}{2} + \epsilon}),
\eeq
where we define
\beq
I_l = \sum_{i,j} \frac{\kappa_{l+1}^{\sWg_{ij}}}{l!}\E\left[\left(\frac{\partial}{\partial \Wgnm_{ij}} \right)^l \left( \left( t_1 \Re(\rWgnm^2)_{ij}  + t_2 \Im (\rWgnm^2)_{ij}  +\frac{t_3}{\sqrt N}\sum_{p, q} \rWg_{pi} \rWg_{jq}\right) e^{\expo} \right) \Big \vert \Omega\right] \\
\eeq
and
\beq
I_1^G = \sum_{i,j} \kappa_{2}^{\sGOE_{ij}}\E\left[\frac{\partial}{\partial \Wgnm_{ij}} \left( \left( t_1 \Re(\rWgnm^2)_{ij}  + t_2 \Im (\rWgnm^2)_{ij}  +\frac{t_3}{\sqrt N}\sum_{p, q} \rWg_{pi} \rWg_{jq}\right) e^{\expo} \right) \Big \vert \Omega\right].
\eeq
In the following, we will evaluate \(I_l\) for \(l=1, 2, 3\) separately. We may omit the conditioning on $\Omega$ for the ease of notation.

\subsubsection{Estimate for \( I_1 - I_1^G \)}

Since \(\kappa_2^{\sGOE_{ij}} = \kappa_2^{\sWg_{ij}} = \frac{1}{N} \) for \(i\neq j\), we only need to consider the contribution from the diagonal entries to \( I_1 - I_1^G \). By the definition of \(I_1\) and \( I_1^G \),
\beq \label{eq:t_derivative_l_1a}
    I_1 - I_1^G = \sum_i (\kappa_{2}^{\sWg_{ii}} - \kappa_{2}^{\sGOE_{ii}}) \E \left[ \frac{\partial}{\partial \Wgnm_{ii}} \left( \left( t_1 \Re(\rWgnm^2)_{ii}  + t_2 \Im (\rWgnm^2)_{ii} + \frac{t_3}{\sqrt N}\sum_{p, q} \rWg_{pi} \rWg_{iq}\right) e^{\expo} \right)\right].
\eeq
From \eqref{eq:1d_rWg}, we find that
\[
    \frac{t_3}{\sqrt N} \frac{\partial}{\partial \Wgnm_{ii}} \sum_{p, q} \rWg_{pi} \rWg_{iq} = O(N^{-\frac12+\epsilon}).
\]
Similarly, it can be checked that all terms in the right-hand side of \eqref{eq:t_derivative_l_1a} involving \(\rWg\) are \( O(N^{-\frac12 +\epsilon}) \). Collecting the terms of order $1$ only, we obtain that 
\beq
\begin{aligned}
    I_1 - I_1^G &= \frac{1}{N} \sum_i (w_2 -2) \E \left[ \left(2t_1 \Re\left((\rWgnm^2)_{ii}\rWgnm_{ii}\right) + 2t_2 \Im \left((\rWgnm^2)_{ii}\rWgnm_{ii}\right)+ (t_1\Re(\rWgnm^2)_{ii} + t_2 \Im(\rWgnm^2)_{ii})^2\right) e^{\expo} \right] \\
    &\qquad \qquad + O(N^{-\frac12+\epsilon}).
\end{aligned}
\eeq
Using the estimate \( |G_{ij} - \delta_{ij}s_1|, |(\rWgnm^2)_{ij} - \delta_{ij}s'_1| \leq N^{-\frac12+\epsilon}\) on \( \Omega_4 \), we conclude that
\beq
\label{eq:t_derivative_l_1}
  I_1 - I_1^G = (w_2 -2) \left(2t_1 \Re (s'_1 s_1) + 2t_2 \Im (s'_1 s_1)+ (t_1\Re(s'_1) + t_2\Im(s'_1))^2\right)\E \left[ e^{\expo} \right] + O(N^{-\frac{1}{2} + \epsilon}).
\eeq

\subsubsection{Estimate for \( I_2 \)}

We decompose \( I_2 \) into
\beq
\begin{aligned}
I_2 &=  \sum_{i,j}\frac{W_3}{2N^{\frac32}} \E \left[\left(\frac{\partial}{\partial M_{ij}} \right)^2 \left( \left( t_1 \Re(\rWgnm^2)_{ij} + t_2 \Im(\rWgnm^2)_{ij} + \frac{t_3}{\sqrt N} \sum_{p, q} \rWg_{pi} \rWg_{jq} \right) e^{\expo}\right)\right] \\
&:= I_{2, 0} + 2 I_{1, 1} + I_{0, 2},
\end{aligned}
\eeq
where
\beq
 I_{r, 2-r} :=  \sum_{i,j}\frac{W_3}{2N^{\frac32}} \E \left[\left(\frac{\partial}{\partial M_{ij}} \right)^r \left( t_1 \Re(\rWgnm^2)_{ij} + t_2 \Im(\rWgnm^2)_{ij} + \frac{t_3}{\sqrt N} \sum_{p, q} \rWg_{pi} \rWg_{jq} \right) \cdot \left(\frac{\partial}{\partial M_{ij}} \right)^{2-r} e^{\expo} \right].
\eeq

We first consider the case \( i \neq j \) in the summand of \( I_{r, 2-r} \) for \( r=0, 1, 2 \). Recall that all terms of \( O(N^{-\frac12+\epsilon}) \) are negligible in the sense that they can be absorbed into the error term in the right-hand side of \eqref{eq:I_l defined}.

\begin{itemize}
  \item For \( I_{2, 0} \), we note that the terms arising from the derivatives of the $\rWgnm^2$ are negligible, which can be checked by following the argument in the proof of Theorem 3.3 in \cite{Lytova2009}, especially the estimate of \(T_3\) in (3.53) of \cite{Lytova2009}. For example, one of such terms is bounded by
\beq
\label{eq:l_1_error}
    \left| N^{-\frac32} \sum_{i,j}\frac{W_3}{2} \E\left[t_1 \Re(\rWgnm_{ii}\rWgnm_{jj}(\rWgnm^2)_{ij}) e^\expo\right]\right| \leq \frac{C}{\eta^4 \sqrt{N}}.
\eeq
To prove it, we consider a vector $\bsu = (\rWgnm_{11}, \rWgnm_{22}, \dots, \rWgnm_{NN})$ and proceed as
\beqq
    \left| \sum_{i, j} \rWgnm_{ii}\rWgnm_{jj}(\rWgnm^2)_{ij} \right| = \left| \langle \overline{\bsu}, \rWgnm^2 \bsu \rangle \right| \leq \| \rWgnm^2 \| \| \bsu \|^2 \leq N \| \rWgnm^2 \| \| \rWgnm \|^2 \leq \frac{N}{\eta^4}.
\eeqq   

On the other hand,
\beq
  \begin{aligned}
      \left(\frac{\partial}{\partial M_{ij}}\right)^2 \rWg_{pi}\hat{\rWgnm}_{jq} &= 
6(\rWg_{pi} \hat{\rWgnm}_{ji}^2 \hat{\rWgnm}_{jq} + \hat{\rWgnm}_{pj}\hat{\rWgnm}_{ii} \hat{\rWgnm}_{ji}\hat{\rWgnm}_{jq} + \hat{\rWgnm}_{pi}\hat{\rWgnm}_{ji}\hat{\rWgnm}_{jj}\hat{\rWgnm}_{iq})\\
  & \qquad + \rWg_{ii}\hat{\rWgnm}_{jj}(4\hat{\rWgnm}_{pi}\hat{\rWgnm}_{jq} + 2\hat{\rWgnm}_{pj}\hat{\rWgnm}_{iq}).
\end{aligned}
\eeq
From the estimate \( |\rWg_{ij} - \delta_{ij}s_2| \leq N^{-\frac{1}{2}+\epsilon} \) on \( \Omega_3 \) the concentration of $\rWg_{\bse\bse}$ on $\Omega_2$, we then claim that
\beq
\begin{aligned} \label{eq:I20estimate}
  I_{2,0} =&  \frac{W_3t_3}{2N^2}\sum_{i,j} \E\left[6 \rWg_{ii}\hat{\rWgnm}_{jj}(\sum_p\hat{\rWgnm}_{pi})(\sum_q\hat{\rWgnm}_{qj})e^{\expo}\right] + O(N^{-\frac{1}{2}+\epsilon})\\
    =&  3W_3t_3 s_2^2\E \left[\rWg_{\bse\bse}^2 e^{\expo}\right] + O(N^{-\frac{1}{2}+\epsilon}) = 3W_3t_3 s_2^4 \E[e^{\expo}] + O(N^{-\frac{1}{2}+\epsilon}).
\end{aligned}
\eeq
All the other terms in \(I_{2, 0}\) arising from \(\left(\frac{\partial}{\partial M_{ij}} \right)^2 \sum_{p,q}\rWg_{pi} \hat{\rWgnm}_{jq}\) are negligible. One of such terms is bounded by
\beq
\label{eq:l_2_error_1}
\begin{aligned}
    \left| \frac{W_3t_3}{2N^2}\sum_{i,j} \E \left[(\sum_p\rWg_{pj})\hat{\rWgnm}_{ii} \hat{\rWgnm}_{ji}(\sum_q\hat{\rWgnm}_{jq} )e^{\expo} \right] \right| 
    \leq  \frac{2|W_3||t_3|}{(\hat{J} - 2)N^{\frac{5}{2} - 3\epsilon}} \sum_{i,j} \E \left[|e^{\expo}|\right] = O(N^{-\frac{1}{2} + 3\epsilon})
\end{aligned}
\eeq
where we use the definitions of $\Omega_1$, \( \Omega_2 \) and \( \Omega_3 \).

\item For \( I_{1, 1} \), the estimates for the negligible terms can be done by using the argument similar to \eqref{eq:l_2_error_1} and \eqref{eq:l_1_error}. The remaining $O(1)$-terms are
\beqq
    \frac{W_3t_3}{N^2}\sum_{i,j}\E\left[\sum_{p,q}\rWg_{pi} \rWg_{jq}\left( t_1\Re \left(\rWgnm_{ii}(\rWgnm^2)_{jj} + \rWgnm_{jj} (\rWgnm^2)_{ii}\right) + t_2 \Im \left(\rWgnm_{ii}(\rWgnm^2)_{jj} + \rWgnm_{jj} (\rWgnm^2)_{ii}\right)\right)e^{\expo}\right].
\eeqq
Using the definitions of $\Omega_2$ and $\Omega_4$, we write
\beq
\begin{aligned} \label{eq:I11estimate}
  I_{1, 1} 
    & = 2W_3t_3 \left(t_1\Re(s_1s'_1) + t_2 \Im(s_1s'_1)\right)\E\left[\rWg_{\bse\bse}^2 e^{\expo}\right]  + O(N^{-\frac{1}{2} + \epsilon}) \\
    & = 2W_3t_3 \left(t_1\Re(s_1s'_1) + t_2 \Im(s_1s'_1)\right)s_2^2\E\left[e^\expo\right] + O(N^{-\frac{1}{2} + \epsilon}).
\end{aligned}
\eeq

\item For \( I_{0, 2} \), from the same analysis as for \( I_{1, 1} \), 
\beq \label{eq:I02estimate}
  I_{0, 2} = 2W_3t_3 \left(t_1\Re(s_1s'_1) + t_2 \Im(s_1s'_1)\right)s_2^2\E\left[e^\expo\right] + O(N^{-\frac{1}{2} + \epsilon}).
\eeq
Again, the estimate can be done in a similar manner.
\end{itemize}

For the case \(i = j\), since there are only \( N \) terms in the summation in \( I_2 \), all terms are negligible due to the priori bounds on \( \|\rWgnm\| \) and $\sum_{p} \rWg_{pi}$. 

Collecting the terms in \eqref{eq:I20estimate}, \eqref{eq:I11estimate}, and \eqref{eq:I02estimate}, we get 
\beq
\label{eq:t_derivative_l_2}
\begin{aligned}
 & \sum_{i,j} \frac{\kappa_{3}^{\sWg_{ij}}}{2!}\E\left[\left(\frac{\partial}{\partial M_{ij}} \right)^2 \left( \left( t_1 \Re(\rWgnm^2)_{ij} + t_2 \Im(\rWgnm^2)_{ij} + \frac{t_3}{\sqrt N} \sum_{p, q} \rWg_{pi} \rWg_{jq} \right) e^{\expo} \right)\right] \\ 
    =& W_3\left[3 t_3 s_2^4  + 6t_1t_3\Re(s_1s'_1)(s_2)^2 + 6t_2t_3\Im(s_1s'_1)(s_2)^2 \right] \E\left[e^{\expo}\right] + O(N^{-\frac{1}{2} +\epsilon}).
\end{aligned}
\eeq

\subsubsection{Estimate for \( I_3 \)}
Note that any term in \( I_3 \) involving \(\rWg\) is negligible due to the extra \(N^{-\frac12}\) factor. Estimating as in the previous subsection, we obtain that
\beq
\label{eq:t_derivative_l_3}
\begin{aligned}
  I_3 =&\sum\frac{\kappa_4^{\sWg_{ij}}}{3!}\E\left[\left(\frac{\partial}{\partial M_{ij}} \right)^3\left( \left( t_1 \Re(\rWgnm^2)_{ij} + t_2 \Im(\rWgnm^2)_{ij }+ \frac{t_3}{\sqrt N} \sum_{p, q} \rWg_{pi} \rWg_{jq} \right) e^{\expo} \right)\right]\\ 
    =& -4(W_4 - 3) \Bigg[t_1 \Re(s_1^3s'_1) + t_2 \Im(s_1^3s'_1) + \left(t_1\Re(s_1s'_1) + t_2\Im(s_1s'_1)\right)^2\Bigg]\E\left[e^{\expo}\right] + O(N^{-\frac{1}{2} + \epsilon})
\end{aligned}
\eeq
We remark that $O(1)$-terms in \( I_3 \) contribute only to the corrections of linear statistics.

\subsubsection{Proof of Theorem \ref{prop:asy_ind_ls_eg1} for general case}

Let
\beq
\begin{aligned}
    \tilde{\expo}(t) =& \expo(t) - (W_2 - 2)(\cos t)^2\left(t_1\Re(s'_1s_1) + t_2\Im(s'_1s_1) + \frac{1}{2}\left(t_1\Re(s'_1) + t_2\Im(s'_1)\right)^2\right)\\
                         & + W_3(\cos t)^3 \left(t_3s_2^4 + 2t_1t_3\Re(s_1s'_1)s_2^2 + 2t_2t_3\Im(s_1s'_1)s_2^2\right) \\
                         & - (W_4 - 3)(\cos t)^4\left(t_1\Re(s_1^3s'_1) + t_2 \Im(s_1^3s'_1) + (t_1\Re(s_1s_1')+ t_2\Im(s_1s_1'))^2\right).
\end{aligned}
\eeq
Then, plugging \eqref{eq:t_derivative_l_1}, \eqref{eq:t_derivative_l_2}, and \eqref{eq:t_derivative_l_3} into \eqref{eq:t_derivative_stein}, we find that   
\beq
    \frac{\dd}{\dd t}\E[e^{\tilde{\expo}} | \Omega] = O(N^{-\frac{1}{2} + \epsilon}),
\eeq
which implies that
\beq
 \E[e^{\tilde{\expo}(0)}| \Omega] = \E[e^{\tilde{\expo}(\frac{\pi}{2})}| \Omega] + O(N^{-\frac{1}{2} + \epsilon}). 
\eeq
Thus,
\beq
\begin{aligned}
    \lim_{N \rightarrow \infty}\E\left[e^\expo(0)\right] =& \lim_{N \rightarrow \infty} \left(\E\left[e^{\expo(0)}|\Omega\right]\P(\Omega) + \E\left[e^{\expo(0)}|\Omega^c\right]\P(\Omega^c) \right) \\
    =& e^{\expo(0) - \tilde{\expo}(0)}\lim_{N \rightarrow \infty} \E[e^{\tilde{\expo}(0)}|\Omega]= e^{\expo(0) - \tilde{\expo}(0)}\lim_{N \rightarrow \infty} \E[e^{\expo(\frac{\pi}{2})}].
\end{aligned}
\eeq
Here we use the fact that \(\Omega\) holds with high probability and \(\tilde{\expo}(\frac{\pi}{2}) = \expo(\frac{\pi}{2})\). We can now conclude that $(\Re \xi_N(z), \Im \xi_N(z), n_N)$ converges to a multivariate Gaussian vector in distribution as \(N \rightarrow \infty\). By direct calculation, we also find that
\beq
  \begin{pmatrix} \xi_N(z)\\ n_N\end{pmatrix} \Rightarrow \mathcal{N} \left(\begin{pmatrix} b(z) \\ -W_3s_2^4 \end{pmatrix}, \begin{pmatrix} V(z_1) & -2W_3s_1s_1's_2^2 \\ -2W_3s_1s_1's_2^2 & \frac{2}{J^2(J^2 - 1)}\end{pmatrix}\right)
\eeq
with \(b(z)\) and \(V(z)\) are defined in Lemma \ref{lemma:ls_full}. 
Now, using \eqref{eq:fluct_1_quad}, we arrive at
\beq
     \begin{pmatrix} \xi_N (z)\\ \chi_N \end{pmatrix} \Rightarrow \mathcal{N} \left(\begin{pmatrix} b(z) \\ \frac{W_3}{J^2}(1 - \frac{1}{J^2}) \end{pmatrix}, \begin{pmatrix} V(z_1) & 2W_3s_1s_1'(1 - \frac{1}{J^2}) \\ 2W_3s_1s_1'(1 - \frac{1}{J^2}) & 2(1 - \frac{1}{J^2})\end{pmatrix}\right).
\eeq
Hence, the asymptotic Gaussianity of $(\mathcal{N}_N^{(2)}(\varphi), \chi_N)$ follows. For \eqref{eq:bigau_mean} and \eqref{eq:bigau_cov}, the mean and the variance of $\mathcal{N}^{(2)}[\varphi]$ is given in Theorem \ref{conj:1}. The limiting covariance is given by 
\beq
    -2W_3(1 - \frac{1}{J^2})\oint_{\Gamma}\varphi(z)s(z)s'(z)\frac{\dd z}{2\pi \ii} = 2W_3(1 - \frac{1}{J^2})\tau_1(\varphi).
\eeq
where we use the change of variables $z \mapsto s$ mapping \(\C \setminus [-2,2]\) to the disk \(|s| < 1\) with \(s + \frac{1}{s} = -z\) and (4.16) in \cite{MR3649446}. This completes the proof of Theorem \ref{prop:asy_ind_ls_eg1} for general case.

\section{Matching}
\label{sec:large_t_behavior}

In the transitional regime, we took $2\beta= \frac1{J} + \frac{B}{\sqrt{N}}$. 
The ferromagnetic regime and the paramagnetic regime correspond to the limiting cases $2\beta>J$ and $2\beta<J$, respectively. 
In this section, we will consider formal limits $B\to \pm \infty$ of the formula given in the main result, Theorem~\ref{thm:fluc_free_energy_para_ferro}, and check the consistency with the results for ferromagnetic and  paramagnetic regimes obtained in \cite{MR3649446}. 

Theorem~\ref{thm:fluc_free_energy_para_ferro} states that the free energy $F_N$ is close to the random variable
\beq    \label{transfor}
        \cFt:= \frac{1}{4J^2} +\frac{\bnp}{2J\sqrt{N}} + \frac{\log N}{4N}  + \frac{\bnp^2 J^2}{4N}
        + \frac{1}{N}\go + \frac1{N} \fe(\gt)
\eeq
in an appropriate sense. Here, $(\go, \gt)$ is a Gaussian vector independent of $B$. The function $\fe(x)$ is given by~\eqref{eq:def_fe}. 
In ferromagnetic and paramagnetic regimes, \cite{MR3649446} shows that the free energy is close to 
\beq \label{cff}
    \cFf:=  \beta \big( J + \frac{1}{J} \big)  - \frac{1}{2} \log (2\beta J) - \frac{1}{4J^2} -  \frac{1}{2}
    +  \frac{\beta- \frac1{2J}}{\sqrt{N}} \mathcal{N}(f_2', \alpha_2')
\eeq
and
\beq \label{cfp}
    \cFp:= \beta^2+ \frac1{N} \mathcal{N}(f_1, \alpha_1),
\eeq
respectively, 
where $\mathcal{N}(f, \alpha)$ denotes a Gaussian distribution of mean $f$ and variance $\alpha$. 
The parameters for the Gaussians are (see (4) of \cite{BaikLee17Erratum} which corrected an error in \cite{MR3649446}) 
\beq \label{fs2alp2} \begin{split}
    f_2' &= W_3(J^{-2} - J^{-4}), \\
    \alpha_2' &= 2 (1-J^{-2}) 
\end{split} \eeq
and (see (1.11) and (1.12) of \cite{MR3649446}; we set $J'=J$)
\beq \label{eq:f1} \begin{split}
    f_1 &= \frac{1}{4} \log(1-4\beta^2) + \beta^2(w_2-2) + 2\beta^4(W_4-3) - \frac{1}{2} \log(1-2\beta J), \\
    \alpha_1 &= -\frac{1}{2} \log (1-4\beta^2) + \beta^2 (w_2-2) + 2\beta^4 (W_4-3).
\end{split} \eeq

The function $\fe(x)$ in~\eqref{transfor} is given by 
\beq \label{feforaa}
    \fe(x)=  \frac{\es(x)}{2(\es(x) - x)} - \frac{\es(x)^2}{4(J^2 - 1)} + \frac{\log(\es(x) - x)}{2} + \log \intc\left(\frac{(\es(x)-x)^2}{J^2 - 1}\right)
\eeq
where (recall the formula~\eqref{eq:def_es}) 
\begin{equation} \label{eq:esfolim} \begin{split}
    \es(x) &= \frac{x - \bnp(J^2 - 1) + \sqrt{(x + \bnp(J^2 - 1))^2 + 4(J^2 - 1)}}{2} .
\end{split} \end{equation}
From the formula, for $x=O(1)$, 
\begin{equation} \label{eq:esfolim} 
    \es(x) = \begin{cases}
    x + \frac{1}{B} + O(B^{-2}) \quad & \text{as $B\to +\infty$,} \\
    -B(J^2-1) - \frac{1}{B} + O(B^{-2}) \quad & \text{as $B\to -\infty$.} 
\end{cases} 
\end{equation}

Note that since we set $2\beta= \frac1{J} + \frac{B}{\sqrt{N}}$ in the transitional regime, 
we regard $B=O(\sqrt{N})$ when we take $B\to \pm\infty$.

\subsection{$B\to +\infty$}

Using~\eqref{eq:esfolim}, we find that  for $x=O(1)$, 
\beq
    Q(x) = \frac{Bx}2 + O(\log B).
\eeq
Hence, since $\go$ does not depend on $B$, we see that as $B=O(\sqrt{N})$ with $B>0$, 
\begin{equation}
        \cFt=  \frac{1}{4J^2} + \frac{\bnp}{2J\sqrt{N}} +  \frac{\bnp^2 J^2}{4N} +   
        \frac{B}{2N} \gt  + O\left(\frac{\log B}{N}\right) + O\left(\frac{\log N}{N}\right).
\end{equation}
where $O(f(B, N))$ represents a random variable $X$ such that the moments of $\frac{X}{f(B, N)}$ are all bounded by constants independent of $B$ and $N$.  

We compare the above formula with the ferromagnetic case~\eqref{cff}. 
If we set $2\beta= \frac1{J} + \frac{B}{\sqrt{N}}$, then 
\beq
    \cFf=   \frac{1}{4J^2} + \frac{\bnp}{2J\sqrt{N}} +  \frac{\bnp^2 J^2}{4N}
    +  \frac{B}{2N} \mathcal{N}(f_2', \alpha_2') + O(N^{-3/2}).
\eeq
We note that (see~\eqref{fs2alp2} and~\eqref{gtmeanvc}) the mean and variance are $f_2'= \Exp[\gt]$ and $\alpha_2'= \Var[\gt]$.
The above formula of $\cFt$ is thus consistent with $\cFf$.

\subsection{$B\to -\infty$}

Consider~\eqref{feforaa}. Recall that $\intc(\alpha)= \sqrt{\frac{4\pi}{\alpha}}(1 + O(\alpha^{-1}))$ as $\alpha\to +\infty$ from~\eqref{eq:ialpha_infty}. 
Hence, if $x=O(1)$ and $s(x)\to \infty$, then 
\beq
    Q(x) =  - \frac{s(x)^2}{4(J^2-1)} + \log \sqrt{\frac{4\pi (J^2-1)}{s(x)}} + \frac12 +O\left( \frac1{s(x)} \right).
\eeq
Using~\eqref{eq:esfolim}, we find that for $x=O(1)$, 
\beq
    Q(x) = - \frac{B^2(J^2-1)}{4}+  \log \sqrt{\frac{4\pi}{|B|}}  + O(B^{-1}).
\eeq
Hence, the two leading terms of $\fe(\gt)$ do not depend on $\gt$. 
Therefore, for $B=O(\sqrt{N})$ with $B<0$, 
\begin{equation} \begin{split}
        \cFt 
        &= \frac14 \left( \frac1{J} + \frac{B}{\sqrt{N}} \right)^2 
        +  \frac{1}{2N} \log \left( \frac{4\pi \sqrt{N}}{|B|} \right) 
        +   \frac{1}{N} \go  + O\left( \frac1{NB} \right). 
\end{split} \end{equation}

On the other hand, in the paramagnetic regime, 
if we set $2\beta= \frac1{J}+ \frac{B}{\sqrt{N}}$ with $B<0$, then the parameters in~\eqref{eq:f1} satisfy (see~\eqref{gomeanvc})
\beq \begin{split}
    f_1&= \frac{1}{4} \log(1-J^{-2}) + \frac1{4J^2}(w_2-2) + \frac1{8J^4}(W_4-3)  - \frac{1}{2} \log\left( \frac{|B|J}{\sqrt{N}} \right)+ O(N^{-1/2}) \\
    &= \Exp[\go]  + \frac12  \log \left( \frac{4\pi \sqrt{N}}{|B|} \right)   + O(N^{-1/2}) 
\end{split} \eeq
and
\beq
    \alpha_1 = -\frac{1}{2}\log(1 - J^{-2}) + \frac{w_2 - 2}{4J^2} + \frac{W_4 - 3}{8J^4}+ O(N^{-1/2}) 
    = \Var[\go]+ O(N^{-1/2}) 
\eeq
Thus, if we set $2\beta= \frac1{J}+ \frac{B}{\sqrt{N}}$ with $B<0$, 
then 
\beq
    \cFp = \frac14 \left( \frac1{J} + \frac{B}{\sqrt{N}} \right)^2 
    +  \frac{1}{2N} \log \left( \frac{4\pi \sqrt{N}}{|B|} \right)  + \frac1{N} \mathcal{N}(\Exp[\go], \Var[\go]) + O(N^{-3/2}) .
\eeq
This is consistent with the formula of $\cFt$.


\end{document}